\documentclass[11pt, a4paper]{article} 

\usepackage{amssymb,amsmath, wasysym, graphicx, epsfig, setspace, wrapfig, comment, pgfplots, theorem}
\usepackage{latexsym,dsfont,amsfonts,amsmath,amssymb}
\usepackage{epsfig,psfrag,color}
\usepackage{enumerate}
\usepackage{dsfont}
\usepackage{mathrsfs}   
\usepackage{url}

\usepackage{algpseudocode,algorithm} 
\algdef{SE}[DOWHILE]{Do}{doWhile}{\algorithmicdo}[1]{\algorithmicwhile\ #1}

\usepackage{mathtools}

\usepackage{hyperref} 

\newcommand{\pd}[3]{\frac{\partial ^{#1} #2}{\partial #3}}

\newcommand{\cost}{\mathrm{cost}}
\newcommand{\supp}{\mathrm{supp}}

\newcommand{\ind}{\mathrm{ind}}

\newcommand{\N}[0]{\mathbb{N}}
\newcommand{\Z}[0]{\mathbb{Z}}
\newcommand{\R}[0]{\mathbb{R}}

\graphicspath{{.}{./Figures/}}

\pgfplotsset{compat=newest}
\pgfplotsset{plot coordinates/math parser=false}
\newlength\figureheight
\newlength\figurewidth

\makeatletter
\newcommand{\BIG}{\bBigg@{3}}
\newcommand{\vast}{\bBigg@{4}}
\newcommand{\Vast}{\bBigg@{5}}
\makeatother

\usepackage[normalem]{ulem}

\newtheorem{theorem}{Theorem}
\newtheorem{lemma}[theorem]{Lemma}

\newtheorem{corollary}[theorem]{Corollary}

\newenvironment{proof}{\begin{trivlist}
    \item[\hskip\labelsep{\bf Proof.}]}{$\hfill\Box$\end{trivlist}}
{\theoremstyle{plain} \theorembodyfont{\rmfamily}

\newtheorem{remark}[theorem]{Remark}

\newtheorem{assumption}{\bf Assumption}

\newcommand{\bsDelta}{{\boldsymbol{\Delta}}}

\newcommand{\bse}{{\boldsymbol{e}}}

\newcommand{\bsalpha}{{\boldsymbol{\alpha}}}
\newcommand{\bsgamma}{{\boldsymbol{\gamma}}}

\newcommand{\bseta}{{\boldsymbol{\eta}}}

\newcommand{\bsnu}{{\boldsymbol{\nu}}}

\newcommand{\bstau}{{\boldsymbol{\tau}}}

\newcommand{\bsrho}{\boldsymbol{\rho}}

\newcommand{\bspsi}{\boldsymbol{\psi}}

\newcommand{\bsq}{{\boldsymbol{q}}}

\newcommand{\bsu}{{\boldsymbol{u}}}

\newcommand{\bsx}{{\boldsymbol{x}}}

\newcommand{\bsy}{{\boldsymbol{y}}}
\newcommand{\bsA}{\boldsymbol{A}}

\newcommand{\bsY}{\boldsymbol{Y}}
\newcommand{\bsz}{{\boldsymbol{z}}}

\newcommand{\bsW}{{\boldsymbol{W}}}

\newcommand{\bszero}{{\boldsymbol{0}}}
\newcommand{\bsone}{{\boldsymbol{1}}}
\newcommand{\bsPhi}{\boldsymbol{\Phi}}

\newcommand{\rd}{\mathrm{d}}

\newcommand{\bbP}{\mathbb{P}} 
\newcommand{\bbR}{\mathbb{R}}

\newcommand{\bbN}{\mathbb{N}}
\newcommand{\bbE}{\mathbb{E}}

\newcommand{\calH}{\mathcal{H}}

\newcommand{\calO}{\mathcal{O}}

\newcommand{\calW}{\mathcal{W}}

\newcommand{\fhat}{\widehat{f}}
\newcommand{\Fhat}{\widehat{F}}
\newcommand{\ftilde}{\widetilde{f}}
\newcommand{\Ftilde}{\widetilde{F}}

\def\R{\mathbb{R}}

\newcommand{\setu}{{\mathrm{\mathfrak{u}}}}

\newcommand{\mask}[1]{{}}
\newcommand{\bigO}{\mathcal{O}}

\definecolor{darkred}{RGB}{139,0,0}
\definecolor{darkgreen}{RGB}{0,100,0}
\definecolor{darkmagenta}{RGB}{170,0,120}
\definecolor{darkpurple}{RGB}{110,0,180}
\definecolor{darkblue}{RGB}{40,0,200}
\definecolor{darkbrown}{rgb}{0.75,0.40,0.15}

\newcommand{\be}{\begin{equation}}
\newcommand{\ee}{\end{equation}}
\newcommand{\bea}{\begin{eqnarray}}
\newcommand{\eea}{\end{eqnarray}}
\newcommand{\beas}{\begin{eqnarray*}}
\newcommand{\eeas}{\end{eqnarray*}}

\def\r2p{{\sqrt{2\pi}}}

\numberwithin{equation}{section}
\numberwithin{theorem}{section}
\numberwithin{algorithm}{section}

\makeatletter
\let\@fnsymbol\@arabic
\makeatother

\title{Analysis of preintegration followed by quasi-Monte Carlo integration for distribution functions and densities}

\date{\today}

\author{Alexander D. Gilbert\footnotemark[1] \and
             Frances Y. Kuo\footnotemark[1] \and
             Ian H. Sloan\footnotemark[1]}

\begin{document}
\maketitle

\footnotetext[1]{School of Mathematics and Statistics, UNSW Sydney, Sydney NSW 2052, Australia.\\
                           \texttt{alexander.gilbert@unsw.edu.au},
                           \texttt{f.kuo@unsw.edu.au},
                           \texttt{i.sloan@unsw.edu.au}
                           }

\begin{abstract}
In this paper, we analyse a method for approximating the
distribution function and density of a random variable  that depends in a
non-trivial way on a possibly high number of independent random variables,
each with support on the whole real line. Starting with the integral
formulations of the distribution and density, the method involves
\emph{smoothing} the original integrand \emph{by preintegration}
with respect to one suitably chosen variable, and then applying a suitable
quasi-Monte Carlo (QMC) method to compute the integral of the resulting
smoother function. Interpolation is then used to reconstruct the
distribution or density on an interval. The preintegration technique is a
special case of \emph{conditional sampling}, a method that has previously
been applied to a wide range of problems in statistics and computational
finance. In particular,  the pointwise approximation studied in this work
is a specific case of the conditional density estimator previously
considered in \emph{L'Ecuyer et al., arXiv:1906.04607}. Our theory
provides a rigorous regularity analysis of the preintegrated function,
which is then used to show that the errors of the pointwise and
interpolated estimators can both achieve nearly first-order convergence.
Numerical results support the theory.
\end{abstract}

\section{Introduction}
 
In this paper, we analyse a numerical method for the computation of the
cumulative distribution function (cdf) or probability density
function (pdf) of a continuous random variable $X$, where $X$ depends in a
nontrivial way on many independent continuous random variables $Y_0, Y_1,
\ldots, Y_d$ with known pdfs, each of which has support on the whole real line.
We write $X = \phi(Y_0, \ldots, Y_d)$, where
often the computation of $\phi$ requires significant resources, for
example, the solution of a partial differential equation.

Starting with the formulation of the cdf at a given point as an expected
value (integral) of an indicator function (see \eqref{eq:cdf-expect}
below), the method consists of \emph{preintegration}
\cite{GKS13,GKLS18} (a process in which one well-chosen variable is first
integrated out, with the aim of producing a relatively smooth function of
the remaining $d$ variables), followed by a suitable quasi-Monte
Carlo (QMC) method to integrate over the remaining variables.  This
approach overcomes the difficult aspect of the original integration
problem, namely, that $\phi$ occurs in the argument of a non-smooth
function (a jump discontinuity in the case of the cdf; a delta function in
the case of the pdf), making the direct application of QMC integration
problematic. Interpolation is then used to reconstruct the cdf on a given
interval. This is useful for applications where one wishes to build a
surrogate of the cdf to be used many times within a larger computational
problem, e.g., when evaluating $\phi$ is computationally expensive
and so computing pointwise approximations is also expensive.

\emph{Smoothing by preintegration} can be considered as a special
case of \emph{conditional sampling} from the statistics and computational
finance literature
\cite{ACN13a,ACN13b,BayB-HTemp21,Glasserman,GlaSta01,Hol11,LEcLem00,WWH17},
whereby a given function has its variance reduced by the operation of
conditioning on or ``hiding'' (integrating out)
partial information, reducing the problem to a conditional expectation.
Preintegration involves integrating out just a single variable, say $Y_0$, which is equivalent to
conditioning on the other remaining variables, $Y_1, Y_2, \ldots, Y_d$.
In particular, the pointwise approximation method in the present paper is a specific case of the
conditional density estimator studied in \cite{LEcPuchBAb19}.

The principal contribution of the present paper lies in a rigorous error
analysis of the method, assuming only properties of the original function
$\phi$ \emph{before} performing preintegration. The first step in the
analysis (a nontrivial one) is to establish regularity properties of the
preintegrated function, where regularity is meant in the sense of the preintegrated function having all its mixed derivatives of first order  being square integrable.

The analysis in \cite{LEcPuchBAb19} is
essentially different, in that it started with an assumption of suitable
smoothness of the resulting integrand \emph{after} preintegration, not the
original function.

The minimal smoothness property assumed in classical QMC analyses on the
unit cube is membership of BVHK --- the class of functions with bounded
variation in the sense of Hardy and Krause. A sufficient condition for
membership of BVHK is that $f$ has continuous, integrable mixed
derivatives of first order. Our analysis and function space setting
\cite{GKS22,NK14} include preintegrated functions which, after
transformation to the unit cube, do \emph{not} belong to BVHK, see
Appendix~\ref{app:1} for further details.

In this paper we assume that the random variables $Y_i$ for $i =
0,\ldots,d$ have as support the whole real line $\bbR$. This
assumption is essential, since the results do not apply if the $Y_i$
have compact support. In Appendix~\ref{app:2} we give a simple
example on the unit cube illustrating that a single step of
preintegration is not sufficient for the preintegrated function to
belong to BVHK, neither does it possess the necessary mixed derivative
smoothness. To achieve the required first-order mixed derivative
smoothness on $[0, 1]^d$, the paper \cite{GKS10} shows that, in general, one must
perform preintegration with respect to $d/2$ different variables.

It is now known \cite{GKS22b} that for the preintegrated
function to have the classical smoothness property assumed in QMC
analyses, it is necessary (as assumed here and in \cite{GKLS18}) that $\phi$
be a monotone function of the chosen preintegration
variable. If that monotonicity assumption does not hold and $d \ge 2$,
then from \cite{GKS22b} the preintegrated function has generically
one-sided square-root or higher-root singularities, which, as shown in Appendix~\ref{app:3}, 
precludes membership of BVHK.

\subsection{The problem and the approach}

Let $X$ be a continuous real-valued random variable, and denote its
cumulative distribution function (cdf) and probability density function
(pdf) by $F \coloneqq F_X$ and $f \coloneqq f_X$, respectively, which are 
defined on an interval $[a, b]$ but which are
not known \emph{a priori}.
Suppose that $X$ is a function of $d + 1$ independent random
variables $Y_0, Y_1, \ldots, Y_d \in \R$,
\begin{equation}
\label{eq:X}
X = \phi(Y_0, Y_1, \ldots, Y_d)\,,
\end{equation}
for some computable (but possibly expensive) function $\phi : \R^{d +
1} \to \R$. Suppose also that the density and distribution functions of
each $Y_i$ are known, and are denoted by $\rho_i$ and $\Phi_i$,
respectively, with $\supp(\rho_i) = \R$.
Realisations of the random variables $X$ and $Y_i$ are denoted by $x$ and $y_i$, respectively.

An example of such a random variable $X$ is the value of an Asian option,
where the random variables $Y_0, \ldots, Y_d$ are the Brownian motion
increments at each time step and are normally distributed. Another example
is the linear functional of the solution of a PDE with a $\log$-normal
random coefficient, where $Y_0, \ldots, Y_d$ are zero-mean normal random
variables that correspond to the random parameters in a series
expansion of a Gaussian random field.

Our goal is to approximate the cdf $F$ and the pdf $f$ (the derivative
of $F$) on a compact interval $[a, b] \subset\bbR$, which is done in two
steps:
\begin{enumerate}
\item Approximate $F$ and $f$ at finitely many points $\{t_m\}_{m
    = 0}^M \subset [a, b]$ using quasi-Monte Carlo (QMC) with
    preintegration (see Sections~\ref{sec:qmc} and~\ref{sec:preint},
    respectively).
\item Reconstruct $F$ and $f$ on $[a, b]$ by interpolating the
    approximations at the points $\{t_m\}_{m = 0}^M$.
\end{enumerate}

The cdf and pdf can each be written as expected values (i.e.,
potentially high-dimensional integrals) with respect to $Y_0,
Y_1, \ldots, Y_d$. For $t \in [a, b]$, we have
\begin{align} \label{eq:cdf-expect}
 F(t)
&=
\bbE \big[ \ind(t - X)\big]
= \int_{\R^{d + 1}}
 \ind\big(t - \phi(y_0, \ldots, y_d )\big)
\Bigg(\prod_{i = 0}^d \rho_i(y_i)\Bigg)
\rd y_0 \cdots \rd y_d\,, \\
\label{eq:pdf-expect}
f(t) &= \bbE \big[ \delta(t - X) \big] =
\int_{\R^{d + 1}} \delta\big(t - \phi(y_0, \ldots, y_d)\big)
\Bigg(\prod_{i = 0}^d \rho_i(y_i)\Bigg) \rd y_0 \cdots \rd y_d\,,
\end{align}
where $\ind(\cdot)$ is the indicator function
\begin{align*}
 \ind(z) =
 \begin{cases}
 1 & \text{if } z \ge 0\,,\\
 0 & \text{otherwise\,,}
 \end{cases}
\end{align*}
and $\delta(\cdot)$ is the Dirac $\delta$ distribution characterised by
the properties
\begin{align} \label{eq:dirac}
 &\delta(z) = 0 \quad \text{for all } z \neq 0,\quad
 \text{and} \nonumber\\
 &\int_{-\infty}^\infty g(z)\, \delta(z) \, \rd z = g(0)
 \quad\mbox{for all sufficiently smooth functions } g\,.
\end{align}

Note that if the pdf exists, then it is the derivative of the cdf,
i.e., $f = F'$. Indeed, the integral in \eqref{eq:pdf-expect}
can be interpreted as differentiating \eqref{eq:cdf-expect} with
respect to $t$ in the distributional sense (recall that $\delta$ is the
distributional derivative of $\ind$). In the following, we will give more
precise conditions we need to impose on the function $\phi$ such that
the integral formulations \eqref{eq:cdf-expect} and \eqref{eq:pdf-expect}
are well defined.

QMC theory alone is unable to tackle integrals such as
\eqref{eq:cdf-expect} and \eqref{eq:pdf-expect}
due to the discontinuity introduced by the indicator and Dirac $\delta$
functions. This discontinuity means that the integrand fails to belong to
the function spaces required for QMC theory. The integrand in the
formulation of the pdf \eqref{eq:pdf-expect} is not even a function, but
rather a distribution that is 0 everywhere except when $\phi(y_0, \ldots,
y_d) = t$. However, recent work \cite{GKLS18} on smoothing by
preintegration was successful in handling simple discontinuities caused by
an indicator function in the integrand, both practically and
theoretically. In this paper, we extend the work to cover
distributions involving a $\delta$ function, as well as extending the
theory for the indicator function.

\subsection{Preintegration}
To explain the idea of preintegration, consider a simple discontinuous
function
\[
  g(y_0, \ldots, y_d) = \ind\big(\phi(y_0, \ldots, y_d)\big)\,,
\]
where the inner function $\phi: \R^{d + 1} \to \R$ is sufficiently smooth
and satisfies certain technical assumptions with respect to a specially
chosen variable, which throughout this paper we shall take to be $y_0$. As stated above, a key assumption is that $\phi$ is strictly
increasing in $y_0$
\begin{equation}\label{eq:monotone}
\frac{\partial\phi}{\partial{y_0}}> 0,
\end{equation}
(cf.~Assumption~\ref{asm:phi} below). To perform the preintegration
step, we integrate with respect to this special variable $y_0$ to
give a \emph{preintegrated} function
\begin{equation}\label{eq:P0}
 P_0 g(y_1, \ldots, y_d)
 := \int_{-\infty}^\infty g(y_0, \ldots, y_d)\, \rho_0(y_0) \, \rd y_0\,,
\end{equation}
which is now a $d$-variate function of the remaining variables $y_1,
\ldots, y_d$.

The key point from \cite{GKLS18} is that if we fix $(y_1, \ldots, y_d) \in
\R^d$ and treat $\phi(\cdot, y_1, \ldots, y_d)$ as a function of the
single variable $y_0$, then since $\phi$ is strictly increasing with
respect to~$y_0$, the discontinuity in $g$ either occurs at a single
point, in which case that variable can then be integrated out, or does not
occur at all. Thus after preintegration, there is no longer any
discontinuity and the result is a $d$-variate function $P_0g$ that under
suitable conditions is as smooth as the original smooth function~$\phi$.
In this way, after performing the preintegration step, either exactly or
numerically, one can use a $d$-dimensional cubature rule for the remaining
dimensions. Due to the smoothness of $P_0 g$, the cubature
error will now converge at a faster rate, e.g., close to $\bigO(1/N)$ for
a QMC method using $N$ points.

The smoothing by preintegration step was recently analysed in
\cite{GKLS18} which extends the earlier work
\cite{GKS13,GKS17note,GKS17}; see also \cite{GKS22b}. In this
paper we follow the terminology of preintegration instead of conditioning,
since \cite{GKLS18} forms the foundation of our theory.

\subsection{Related work and other approaches}

The recent paper \cite{BAbdLEcOwenPuch21} used QMC to construct kernel
density estimators and histograms.
Their need to balance the variance and bias means the
provided error convergence rate deteriorates very rapidly with
dimension. The paper \cite{GNR15} replaced the non-smooth functions in
\eqref{eq:cdf-expect} and \eqref{eq:pdf-expect} by smooth approximations
and then applied multilevel Monte Carlo methods. The paper
\cite{BotSalMac19} introduced smooth cdf and pdf estimators for the
specific case of a sum of dependent lognormals, with promising numerical
results using QMC.

As stated earlier,
preintegration is a specific case of conditional sampling or conditioning,
a method widely used in the statistics and computational
finance literature
\cite{ACN13a,ACN13b,BayB-HTemp21,Glasserman,GlaSta01,Hol11,LEcLem00,WWH17}.
The idea to combine a QMC rule with conditioning was first presented in
\cite{LEcLem00}, where it was applied to compute probabilities for a
stochastic activity network. More recently, \cite{Asm18} used conditional
Monte Carlo for density estimation in the specific case of a sum of random
variables. Then, the paper \cite{LEcPuchBAb19} introduced a general conditional
density estimator (CDE) and \cite{PeFuHuLEcTuf21} used conditioning for
variance reduction for a generalised likelihood ratio density estimator,
where both works used Monte Carlo and QMC integration. The pointwise
estimator analysed in this paper is a special case of the QMC CDE in
\cite{LEcPuchBAb19}. That paper also presented a more general CDE, by
allowing conditioning on general sets, as well as considering convex
combinations of CDE's and using QMC within a generalised likelihood ratio
density estimator. To analyse the CDE with QMC, the paper
\cite{LEcPuchBAb19} assumed that after conditioning the result belonged to
the class BVHK,
which allowed the use of the classical Koksma--Hlawka inequality to
bound the QMC error.

\section{Mathematical background}

In this section, we introduce the required background material on
preintegration, QMC, and the function spaces that we need.

We start with some notation. Recall that $y_0$ is the special variable
with respect to which we perform preintegration. We denote the remaining
$d$ variables by $\bsy = (y_1, \ldots, y_d) \in \R^d$ and all of the $d +
1$ variables collectively by $\bsy_{0:d} \coloneqq (y_0, \ldots, y_d) \in
\R^{d + 1}$ or $(y_0, \bsy)$. Similarly, we denote the products of
univariate functions $(\rho_i)_{i = 0}^d$, by
\[
\bsrho(\bsy) \coloneqq \prod_{i = 1}^d \rho_i(y_i)
\quad \text{and} \quad
\bsrho_{0:d}(y_0, \bsy) \coloneqq \prod_{i = 0}^d \rho_i(y_i)
= \rho_0(y_0)\,\bsrho(\bsy)\,.
\]

Let $\N_0 \coloneqq \{0 , 1, 2, \ldots\}$ and $\N \coloneqq \{1, 2,
\ldots\}$ denote the set of natural numbers with and without zero,
respectively. Let $\{0:d\} \coloneqq \{0,1, \ldots, d\}$ and define
$\{1:d\}$ analogously. Let $\bsnu = ({\nu_0, \nu_1}, \ldots, \nu_d) \in
\N_0^{d + 1}$ be a multi-index, and let $|\bsnu| \coloneqq \sum_{i = 0}^d
\nu_i$ denote its \emph{order} and $\supp(\bsnu) \coloneqq \{j \in \{0:d\}
: \nu_j> 0\}$ denote its \emph{support}. Operations and relations between
multi-indices are defined componentwise, e.g., for $\bseta, \bsnu \in
\N_0^{d + 1}$ we write $\bseta \leq \bsnu$ if and only if $\eta_i \leq
\nu_i$ for all $i = 0, 1, \ldots, d$, and addition is defined by $\bseta +
\bsnu = (\eta_i + \nu_i)_{i = 0}^d$. For $\bsy_{0:d}\in\bbR^{d+1}$
and $\bsnu \in \N_0^{d + 1}$, we denote the \emph{active} variables by
$\bsy_\bsnu \coloneqq (y_i : \nu_i > 0)_{i=0}^d$ and the \emph{inactive}
variables by $\bsy_{-\bsnu} \coloneqq (y_i : \nu_i = 0)_{i=0}^d$.
Analogously to the notation $(y_0, \bsy)$, we denote the $(d +
1)$-dimensional concatenation of $\nu_0 \in \N_0$ and $\bsnu = (\nu_1,
\nu_2, \ldots, \nu_d) \in \N_0^d$ by $(\nu_0, \bsnu) = (\nu_0, \nu_1,
\nu_2, \ldots, \nu_d) \in \N_0^{d+1}$.

\subsection{Function spaces}
\label{sec:function_spaces}

Here we introduce our function space setting. Although we deal with both
$(d+1)$- and $d$-variate functions throughout this paper, we give
definitions only for the $(d + 1)$-variate spaces, since the $d$-variate
spaces can be defined analogously by simply excluding the variable $y_0$.

We begin by defining some shorthand notation for mixed partial
derivatives. For $i = 0, 1, \ldots, d$ and a multi-index $\bsnu  \in
\N_0^{d + 1}$, let
\[
 D^i \coloneqq \pd{}{}{y_i}
 \qquad\mbox{and}\qquad
 D^\bsnu = \prod_{i = 0}^d \pd{\nu_i}{}{y_i^{\nu_i}}
\]
denote the first-order derivative and the higher order mixed derivative of
order $\bsnu$, respectively. This notation will also be used for weak
derivatives, where the $\bsnu$th weak derivative of $g$ is defined to be
the distribution $D^\bsnu g$ that satisfies
\begin{align*}
 \int_{\R^{d + 1}} &D^\bsnu g(\bsy_{0:d})\, v(\bsy_{0:d}) \,\rd \bsy_{0:d}
 = (-1)^{|\bsnu|} \int_{\R^{d + 1}} g(\bsy_{0:d})\, D^\bsnu v(\bsy_{0:d}) \,\rd \bsy_{0:d}
\end{align*}
for all $v \in C^\infty_0(\R^{d + 1})$. Here $C^\infty_0(\R^{d + 1})$ is
the space of infinitely differentiable functions with compact support.

Let $C(\R^{d + 1})$ denote the space of continuous functions on $\R^{d +
1}$. For $\bsnu \in \N_0^{d + 1}$ let $C^\bsnu(\R^{d + 1})$ denote the
space of functions with continuous \emph{mixed} derivatives up to $\bsnu$:
\[
C^\bsnu(\R^{d + 1}) \coloneqq \big\{g \in C(\R^{d + 1}) :
D^{\bseta} g \in C(\R^{d + 1})
\text{ for all } {\bseta} \leq \bsnu\}\,.
\]

To provide a function space setting for $\phi$ in \eqref{eq:X}, we
introduce a class of Sobolev spaces of dominating mixed smoothness on
$\R^{d + 1}$, where the behaviour of derivatives as $y_i \to \pm \infty$ is
controlled by functions different from the densities~$\rho_i$. To this
end, for $i=0,1,\ldots,d$, let $\psi_i : \R \to \R$ be a strictly positive
and integrable \emph{weight function}. We denote the whole collection of
weight functions by $\bspsi = (\psi_i)_{i = 0}^d$. Also, let $\bsgamma
\coloneqq \{\gamma_\setu > 0 : \setu \subseteq \{0:d\}\}$ be a collection
of positive real numbers called \emph{weight parameters}; they model the
relative importance of different collections of variables, i.e.,
$\gamma_\setu$ describes the relative importance of the collection of variables
$(y_i : i \in \setu)$. We set $\gamma_\emptyset \coloneqq 1$.

Then for $\bsnu \in \N_0^{d + 1}$, define the \emph{Sobolev space of
dominating mixed smoothness} of order~$\bsnu$, denoted by $\calH^\bsnu_{d
+ 1}$, to be the space of locally integrable functions on $\R^{d + 1}$
such that the norm
\begin{equation*}
\|g\|_{\calH^\bsnu_{d + 1}}^2 \coloneqq
\sum_{\bseta \leq \bsnu} \frac{1}{\gamma_\bseta}
\int_{\R^{d  +1}} |D^{\bseta} g(\bsy_{0:d})|^2\,
\bspsi_{\bseta}(\bsy_\bseta)\,\bsrho_{-\bseta}(\bsy_{-\bseta})
 \,\rd \bsy_{0:d}
\end{equation*}
is finite, where we introduced the shorthand notations
\[
  \gamma_\bseta \coloneqq \gamma_{\supp(\bseta)}, \quad
  \bspsi_{\bseta}(\bsy_\bseta) \coloneqq \prod_{i=0,\,\eta_i \neq 0}^d \psi_i(y_i)
  \quad\mbox{and}\quad
  \bsrho_{-\bseta}(\bsy_{-\bseta}) \coloneqq \prod_{i=0,\,\eta_i = 0}^d \rho_i(y_i).
\]

Recall from the Introduction that we plan to carry out preintegration on a
non-smooth function of $d+1$ variables with appropriate properties, with
the aim of obtaining a smooth function of $d$ variables. We therefore need
(but do not write down) an analogous $d$-variate Sobolev space
$\calH_d^\bsnu$ with variables indexed from $1$ to $d$.

An important property of the Sobolev space of \emph{first-order}
dominating mixed smoothness, i.e., $\calH^{\bsone}_d$ with $\bsone
\coloneqq (1, 1, \ldots, 1)$, is that it is equivalent to the (unanchored)
ANOVA space over the unbounded domain $\R^d$ introduced in \cite{NK14}.
Explicitly, it was shown recently in \cite{GKS22} that if the weight
functions $\psi_i$ satisfy
\begin{equation}
\label{eq:psi}
\int_{-\infty}^\infty \frac{\Phi_i(z)(1 - \Phi_i(z))}{\psi_i(z)} \, \rd z \,< \, \infty
\quad \text{for all } i = 1, 2, \ldots, d,
\end{equation}
then $\calH^{\bsone}_d$ and the ANOVA space from \cite{NK14} are
equivalent. This equivalence is crucial for our analysis, because it
immediately shows that the bounds on the QMC error from \cite{NK14} also
hold in $\calH^{\bsone}_d$ (see Theorem~\ref{thm:qmc} below). Since the
preintegration theory in \cite{GKLS18} is formulated in $\calH^\bsnu_d$,
without this equivalence, there would be a mismatch between the
settings for the analysis of preintegration and QMC methods. With the
equivalence established, we can from now on work exclusively with the
spaces $\calH_{d}^\bsnu$, and have no need to introduce the ANOVA space
from \cite{NK14}. Note that the condition \eqref{eq:psi} is also
assumed throughout \cite{NK14}, where it is required for the QMC error
bounds to hold. Examples of common pairings $(\rho_i, \psi_i)$ satisfying
\eqref{eq:psi} can be found in \cite[Table~3]{KSWW10}. 
Note that $\psi_i^2$ in \cite{KSWW10,NK14} is replaced here, and in \cite{GKS22},
by $\psi_i$.
We assume \eqref{eq:psi} holds throughout.

\subsection{Quasi-Monte Carlo methods}
\label{sec:qmc}

In the classic case of the unit cube, an $N$-point QMC approximation
(see e.g., \cite{DKS13,Nie92}) for the integral of a function $g: [0,1]^d\to \bbR$ is given
by
\[
\frac{1}{N} \sum_{n=0}^{N - 1} g(\bsq_n) \approx
\int_{[0, 1]^d} g(\bsu) \, \rd \bsu \,,
\]
where here the cubature points $\{\bsq_n\}_{n=0}^{N - 1} $ are
deterministically chosen to be well-distributed within $[0, 1)^d$, and to
have desirable approximation properties.

In this paper, we consider a simple class of randomised QMC methods called
\emph{randomly shifted rank-1 lattice rules}, for which the QMC points are
given by
\begin{equation}
\label{eq:lattice}
 \bsq_n = \bigg\{ \frac{n \bsz}{N} + \bsDelta \bigg\}
\quad \text{for } n = 0, 1, \ldots, N - 1.
\end{equation}
Here $\bsz \in \{1, 2, \cdots, N - 1\}^d$ is the \emph{generating vector},
$\bsDelta \in [0, 1)^d$ is a uniformly distributed random shift and
$\{ \cdot \}$ denotes taking the fractional part of each component.

The benefits of randomly shifting the point set are threefold: (i)
the resulting approximation is unbiased; (ii) we can take the average
of the approximations from a small number of i.i.d.\ random shifts as the
final approximation and use the sample variance to estimate the
mean-square error; and (iii) for functions in $\calH_d^\bsone$,
randomly shifted lattice rules with good $\bsz$ can be constructed
efficiently (see below) to achieve nearly $\calO(N^{-1})$ convergence
of the root-mean-square error (RMSE).

To approximate an integral over an unbounded domain, one must map the
point set $\{\bsq_n\}_{n=0}^{N - 1}$ from the unit cube to $\R^d$. In the
case of an integral with respect to a product of densities,
as we have in \eqref{eq:cdf-expect}, we can perform this
mapping by applying the inverse cdf componentwise.
An $N$-point QMC approximation for the integral of a
function $g: \bbR^d\to\bbR$ is then given by
\begin{equation}
\label{eq:qmc}
Q_{d, N} (g)\coloneqq \frac{1}{N} \sum_{n=0}^{N - 1} g(\bsPhi^{-1}(\bsq_n))
\approx \int_{[0,1]^d} g(\bsPhi^{-1}(\bsu)) \,\rd \bsu
 = \int_{\R^d} g(\bsy)\,\bsrho(\bsy) \, \rd \bsy\,,
\end{equation}
where $\bsPhi^{-1}$ denotes the application of the inverse cdf $\Phi_i^{-1}$
in each dimension $i$, recalling that $\Phi_i$ is the cdf of the density
$\rho_i$. For the remainder of the paper we only consider approximating
integrals on $\R^d$, and so we denote the transformed QMC points by
\begin{align} \label{eq:transQMC}
 \bstau_n = \bsPhi^{-1}(\bsq_n) \in \R^{d}
 \quad \text{for } n = 0, 1, \ldots, N - 1.
\end{align}

It was proved in \cite{NK14} that good generating vectors $\bsz$ for the
approximation \eqref{eq:qmc} can be constructed using a
component-by-component (CBC) algorithm to achieve almost the optimal
convergence rate {for the RMSE} in a certain \emph{first-order ANOVA space} (which as we
have discussed is equivalent to $\calH^\bsone_d$). Below, we restate the
error bound from \cite{NK14}, but now in terms of $\calH_d^{\bsone}$
rather than the equivalent ANOVA space used in~\cite{NK14}.

\begin{theorem} \label{thm:qmc}
Suppose \eqref{eq:psi} holds. Let $\omega \in (1,2]$ and $c <\infty$ be
such that
\begin{equation}
\label{eq:omega}
 \frac{1}{\pi^2k^2} \int_{-\infty}^\infty \frac{\sin^2(\pi\,k\, \Phi_i(y))}{\psi_i(y)}  \, \rd y
\,\leq\, \frac{c}{|k|^{\omega}}
\quad \text{for all } k \in \Z \setminus \{0\} \mbox{ and all } i = 1, \ldots, d\,.
\end{equation}
Let $N \in \N$ and suppose that  $\bsz$ is a generating vector constructed
using the CBC algorithm from \cite{NK14}. Then for $g \in
\calH_d^{\bsone}$, the RMSE (with the expectation
taken with respect to the random shift $\bsDelta$) of the randomly
shifted lattice rule approximation \eqref{eq:qmc} corresponding to $\bsz$
satisfies
\begin{equation}
\label{eq:cbc-err}
\sqrt{\bbE_\bsDelta \bigg[\bigg|\int_{\R^d} g(\bsy) \bsrho(\bsy) \, \rd \bsy - Q_{d, N} (g)\bigg|^2\bigg]}
\le C_{d, \bsgamma, \lambda}\,[\phi_\mathrm{tot}(N)]^{-1/(2\lambda)}\,
\|g\|_{\calH_d^{\bsone}}
\end{equation}
for all $\lambda \in (1/\omega, 1]$, with
\[
 C_{d, \bsgamma, \lambda} := \bigg(\sum_{\bszero \neq \bseta \in \{0, 1\}^d} \gamma_\bseta^\lambda\,
 \big[2\, c\, \zeta ( \omega \lambda)\big]^{|\bseta|}\bigg)^{1/(2\lambda)},
\]
where $\phi_\mathrm{tot}$ is the Euler totient function and $\zeta$ is the
Riemann zeta function.
\end{theorem}
\begin{proof}
Let $\calW_d$ denote the ANOVA space from \cite{NK14}. Theorem 8 from
\cite{NK14} gives the error bound \eqref{eq:cbc-err} for $g \in
\calW_d$ and with the $\calW_d$-norm on the right. Since
\cite[Theorem~13]{GKS22} shows that $\calW_d$ and $\calH_d^\bsone$ are
equivalent under assumption \eqref{eq:psi} on the weight functions, with
$\|g\|_{\calW_d} \leq \|g\|_{\calH_d^{\bsone}}$, the result is proved.
\end{proof}

Observe that the convergence rate $N^{-1/(2\lambda)}$ of the RMSE
is governed by the parameter $\omega \in (1,2]$ from
\eqref{eq:omega}, which in turn depends on the interaction between the
pairs $(\rho_i,\psi_i)$. Ideally, we would like $\omega$ to be
arbitrarily close to 2, which would allow us to take $\lambda$ arbitrarily
close to $1/2$, giving a convergence rate arbitrarily close to $1/N$.
However, this is not always possible.
Table~3 in \cite{KSWW10} gives values of $r_2 = \omega/2$ such that
\eqref{eq:omega} holds for several pairs $(\rho_i,\psi_i)$, and in
particular, it provides common examples for which $\omega \approx 2$.
Note that, as before, $\psi_{i}$ in this paper is $\psi_{i}^{2}$ in \cite{KSWW10}.
As an example, if each $\rho_i$ is a standard normal 
density, then one can take either $\psi_i(y) = e^{-|y|}$ or $\psi_i = 1/(1 + |y|)$,
resulting in $\omega \approx 2$.
Alternatively, one can take a scaled normal density,
$\psi_i(y) = e^{-y^2/(2\eta)}$ for $\eta > 1$, giving $ \omega = 2(1 - 1/\eta)$,
which for $\eta$ sufficiently large will give a convergence rate close to $1/N$.
In general, the essential feature is that $\psi_i$ decays more slowly than $\rho_i$.

\subsection{Lagrange interpolation in one dimension} \label{sec:1D-approx}

There are two steps to the cdf and pdf estimation algorithms: pointwise
approximation, which we do using a QMC rule after a preintegration step,
and then interpolation on the interval $[a, b]$. For the latter step, we
use Lagrange interpolation based on Chebyshev points in $[a, b]$.

Let $\{t_m\}_{m = 0}^M$ be a collection of distinct points in $[a, b]$ and
let $V_M$ denote the set of all polynomials up to degree $M$ on $[a, b]$.
The \emph{Lagrange interpolation operator} $L_M : C[a, b] \to V_M$ is
given by
\begin{equation}
\label{eq:L_M}
L_M g \coloneqq \sum_{m = 0}^M g(t_m)\, \chi_{M,m},
\qquad
\chi_{M, m}(t) \coloneqq \prod_{\substack{\ell = 0\\ \ell \neq m}}^M
\frac{t - t_\ell}{t_m - t_\ell}\,.
\end{equation}

We now state the classical error bounds for Lagrange interpolation based
on Chebyshev points from, e.g., \cite{Tref20}. For $\sigma \in \N$, let
$W^{\sigma, \infty}[a, b]$ denote the Sobolev space of functions on $[a,
b]$ with essentially bounded derivatives up to order $\sigma$, which we
equip with the norm $\|g\|_{W^{\sigma, \infty}} \coloneqq \max_{q = 0, 1,
\ldots, \sigma} \|g^{(q)} \|_{L^\infty}$.
Let $\sigma \in \N$ and suppose that $g \in W^{\sigma + 1, \infty}[a, b]$.
Then for $M > \sigma$, the error of the Lagrange interpolant based on
Chebyshev nodes satisfies
\begin{equation}
\label{eq:interp_err}
\|g - L_M g\|_{L^\infty} \le
\frac{4\,\|g^{(\sigma + 1)}\|_{L^1}}{\pi \sigma (M - \sigma)^\sigma}\,.
\end{equation}
The original result \cite[Theorem~7.2]{Tref20} was stated in terms of the
\emph{total variation} of $g^{(\sigma)}$ on $[a, b]$, which for $g \in
W^{\sigma + 1, \infty}[a, b]$ is given by $\|g^{(\sigma + 1)} \|_{L^1}$.
As we will see in Section~\ref{sec:regularity}, under our assumptions on
$\phi$, the cdf and pdf are smooth enough to take $\sigma$ up to $d - 1$.

One may also use other methods to approximate $F$ and $f$ on $[a, b]$,
such as splines or best polynomial approximation, but we do not pursue
those directions here.

\section{Smoothing by preintegration} \label{sec:preint}

As explained in the introduction, smoothing by preintegration is a method
of smoothing a discontinuous or kink function by integrating out a single,
specially chosen variable. It is a special case of conditional sampling.
For notational convenience we take $y_0$ to be
this special variable. In this section, we formalise the preintegration
step for indicator functions by following \cite{GKLS18}, and then extend
the method to simple distributions involving $\delta$ distributions, which
will allow us to also apply the preintegration technique to approximate
the pdf as formulated in \eqref{eq:pdf-expect}.

First, we make the following assumptions about the function $\phi$ in
\eqref{eq:X}.

\begin{assumption} \label{asm:phi}
For $d\ge 1$ and $\bsnu\in \bbN_0^d$, let $\phi : \bbR^{d+1}\to\bbR$
satisfy
\begin{enumerate}
\label{itm:mono}\item $D^0\phi(y_0, \bsy) > 0$ for all $ (y_0, \bsy) \in \R^{d + 1}$;
    and
\item for each $\bsy \in \R^d$, $\phi(y_0, \bsy) \to \infty$ as $y_0
    \to \infty$; and
\item $\phi \in \calH^{(\nu_0, \bsnu)}_{d + 1} \cap C^{(\nu_0,
    \bsnu)}(\R^{d + 1})$, where $\nu_0 \coloneqq |\bsnu|+1$.
\end{enumerate}
Additionally, suppose that $\rho_0 \in C^{|\bsnu|}(\bbR)$.
\end{assumption}

It was unresolved from the analysis in \cite{GKS13,GKS17,GKLS18} whether
the \emph{monotonicity} assumption (item~\ref{itm:mono} above) is
necessary. Recently it was proved in \cite{GKS22b} that this is indeed
necessary: if it fails then there may remain a singularity after
preintegration.

\subsection{Smoothing by preintegration for indicator functions}

Motivated by the cdf \eqref{eq:cdf-expect}, for $t\in [a,b]$ we define a
discontinuous function $g_t: \R^{d + 1} \to \R$ by
\begin{equation}
\label{eq:g-ind}
 g_t(y_0, \bsy) \coloneqq \ind\big(t-\phi(y_0, \bsy)\big)\,,
\qquad y_0\in\bbR\,,\;\bsy\in\bbR^d\,,
\end{equation}
where $\phi:\bbR^{d+1}\to\bbR$ satisfies Assumption~\ref{asm:phi}. Note
that \cite{GKLS18} considered functions of the more general form
$g(y_0,\bsy) = \theta(y_0,\bsy)\, \ind(\phi(y_0,\bsy))$, where both
$\theta$ and~$\phi$ are sufficiently smooth, a formulation that allows for
more general discontinuities and kinks. However, since we are here only
concerned with computing probabilities using \eqref{eq:cdf-expect}, the
restricted form in \eqref{eq:g-ind} is sufficient. Also, note that we
consider here shifted indicator functions $\ind(t - \cdot)$ instead of
$\ind(\cdot)$ as in \cite{GKLS18}. This results only in minor changes in
the presentation, and does not affect any of the theory.

Since in Assumption~\ref{asm:phi} we assume that $\phi(\cdot, \bsy)$ is
strictly increasing with respect to $y_0$ for fixed $\bsy \in \R^d$, and
also tends to $\infty$ as $y_0 \to \infty$, there are only two possible
scenarios: either the discontinuity of $g_t(\cdot,\bsy)$ for fixed $t\in
[a,b]$ and $\bsy \in \R^d$ occurs at a unique point in dimension $0$, or
the discontinuity does not occur at all because $\phi(y_0,\bsy)>t$ for all
$y_0 \in \R$.

To describe the former case more explicitly, for given $t \in [a,b]$, we
define the set of $\bsy \in \R^d$ for which the discontinuity occurs by
\begin{equation}
\label{eq:Ut}
 U_t \coloneqq \big\{\bsy \in \R^d : \phi(y_0, \bsy) = t
\text{ for some } y_0 \in \R\big\}\,,
\end{equation}
which (unlike in \cite{GKLS18}) now depends on the point $t$.
Since $\supp (\rho_0) = \R$, the point $y_0$ in \eqref{eq:Ut} is in the support of $\rho_0$.
We have the following equivalence
\[
\bsy \in U_t \; \iff \; t \in \phi(\R, \bsy)\,,
\]
where with a slight abuse of notation we use $\phi(\R, \bsy)$ to denote
the image of $\R$ under $\phi(\cdot, \bsy)$. For $t \in [a, b]$ and $\bsy
\in U_t$, the point at which the discontinuity occurs in dimension~$0$ is
denoted by $\xi(t, \bsy)$, i.e., $\xi(t, \bsy) \in \R$ is the unique real
number such that
\begin{equation}
\label{eq:xi}
 \phi(\xi(t, \bsy), \bsy) = t\,.
\end{equation}
Here uniqueness follows from the monotonicity condition
Assumption~\ref{asm:phi} item~\ref{itm:mono}. Similarly, because $\phi(\cdot,\bsy)$ is
increasing, it follows from \eqref{eq:xi} that for $\bsy \in U_t$,
\begin{equation}
\label{eq:phi-xi-increasing}
\phi(y_0,\bsy) < t \; \iff\; y_0 < \xi(t,\bsy).
\end{equation}

The following Implicit Function Theorem adapted from \cite{GKS13} shows
that $\xi$ is a well-defined function of both $t$ and $\bsy$, and implies
that $\xi$ ``inherits'' the smoothness of~$\phi$.

\begin{theorem}
\label{thm:implicit}
Let $d \geq 1$, $\bsnu \in \N_0^d$, and $[a, b] \subset \R$.
Suppose that $\phi$ and $\rho_0$ satisfy Assumption~\ref{asm:phi}, and
define
\begin{equation}
\label{eq:V}
V \coloneqq \big\{(t,\bsy) \in (a, b) \times \R^d : \phi(y_0, \bsy)  = t
\text{ for some } y_0 \in \R\big\} \,\subset\, [a, b] \times \R^d .
\end{equation}
If $V$ is not empty, then there exists a unique function $\xi \in C^{(\nu_0, \bsnu)}(\overline{V})$
satisfying \eqref{eq:xi} for all $(t, \bsy) \in \overline{V}$.
Furthermore, for $(t, \bsy) \in V$ the first-order derivatives are given by
\begin{align}
\label{eq:Di-xi}
  D^i \xi(t,\bsy) = \frac{\partial}{\partial y_i} \xi(t,\bsy)
  &= - \frac{D^i \phi(\xi(t,\bsy),\bsy)}{D^0 \phi(\xi(t,\bsy),\bsy)}
  \qquad\mbox{for all } i=1,\ldots,d, \quad \text{and}
  \\
  \label{eq:dt-xi}
  \frac{\partial}{\partial t} \xi(t,\bsy)
  &= \frac{1}{D^0 \phi(\xi(t,\bsy),\bsy)}.
\end{align}
\end{theorem}

\begin{proof}
The result $\xi \in C^{(\nu_0, \bsnu)}(\overline{V})$ follows by applying
\cite[Theorem~2.3]{GKS13} to the function $\phi(y_0, \bsy) - t$, with $j =
1$ along with the variables labelled as $x_{i + 1} = y_{i}$ for $i = 0, 1,
\ldots, d$ and $x_{d + 2} = t$. Since the proof of
\cite[Theorem~2.3]{GKS13} is based on a local argument about an arbitrary
point $\bsx$, the restriction $x_{d + 2} = t \in (a, b)$, instead of $\R$,
does not affect the result. Additionally, the original proof was conducted
for the isotropic smoothness spaces, but it can easily be extended to the
dominating mixed smoothness space $C^{(\nu_0, \bsnu)}(\overline{V})$.
Finally, differentiating \eqref{eq:xi} with respect to $y_i$ leads to the
formula \eqref{eq:Di-xi}. Similarly, differentiating \eqref{eq:xi} with
respect to $t$ implies \eqref{eq:dt-xi}.
\end{proof}

With $P_0$ the preintegration operator defined by \eqref{eq:P0}, we now
apply preintegration to the function $g_t$ defined by \eqref{eq:g-ind} for
$t\in [a,b]$, obtaining
\begin{equation*}
 P_0 g_t(\bsy) = \int_{-\infty}^\infty \ind\big(t-\phi(y_0, \bsy)\big)\, \rho_0(y_0) \, \rd y_0\, .
\end{equation*}
From the definition of $\xi(t, \bsy)$ in \eqref{eq:xi} and the
relation \eqref{eq:phi-xi-increasing}, for $\bsy \in U_t$ we can write
\begin{align} \label{eq:P0g-ind}
 P_0 g_t(\bsy)
 = \int_{-\infty}^{\xi(t, \bsy)} \rho_0(y_0) \, \rd y_0
 = \Phi_0(\xi(t, \bsy))\,,
\end{align}
while for $\bsy \in \R^d \setminus U_t$ we have $P_0 g_t \equiv 0$. In
both cases there is no longer any discontinuity.

The main result from \cite[Theorem~3]{GKLS18} showed that if $\phi$
satisfies Assumption~\ref{asm:phi}, along with some extra technical
conditions in Assumption~\ref{asm:phi-h} below, then the preintegrated
function is \emph{as smooth as $\phi$}. The technical conditions are
required to control all of the terms that arise when differentiating $P_0
g_t $ using the chain rule.

As a first illustration, for any $i\in \{1:d\}$ we have
\begin{align} \label{eq:first-der}
  D^i [P_0 g_t(\bsy)]
  = D^i [\Phi_0(\xi(t,\bsy))]
  = \rho_0(\xi(t,\bsy))\,D^i\xi(t,\bsy)
  = -\frac{\rho_0(\xi(t,\bsy))\,D^i\phi(\xi(t,\bsy),\bsy)}{D^0\phi(\xi(t,\bsy),\bsy)},
\end{align}
where we used \eqref{eq:Di-xi}. This motivates the general form of
functions in Assumption~\ref{asm:phi-h} below. Our assumption is
formulated differently from \cite{GKLS18} because we need to account for
the $t$ dependence in this paper and we also aim to give a tight estimate
on the number of terms that arise from the differentiation. This allows us
in Theorem~\ref{thm:preint_ind} below to extend \cite[Theorem~3]{GKLS18}
by providing an explicit bound on the norm. We use $\bse_i$ to denote a
multi-index whose $i$th component is $1$ and all other components are $0$.

\begin{assumption} \label{asm:phi-h}
Let $d\ge 1$, $\bsnu\in \bbN_0^d$, $[a,b]\subset\bbR$, and suppose that
$\phi$ and $\rho_0$ satisfy Assumption~\ref{asm:phi}. Recall the
definitions of $U_t$, $\xi$ and $V$ in \eqref{eq:Ut}, \eqref{eq:xi} and
\eqref{eq:V}, respectively.
Given $q\in\bbN_0$ and $\bseta\le\bsnu$ satisfying $|\bseta|+q \le
|\bsnu|+1$, we consider functions $h_{q,\bseta}: \overline{V} \to \R$ of
the form
\begin{align} \label{eq:h-form}
 \begin{cases}
 h_{q, \bseta}(t, \bsy) = h_{q,\bseta,(r,\bsalpha,\beta)}(t,\bsy)
 \coloneqq
 \displaystyle\frac{(-1)^r \rho_0^{(\beta)}(\xi(t,\bsy))\,
       \prod_{\ell=1}^{r} D^{\bsalpha_\ell} \phi(\xi(t,\bsy),\bsy)}
      {[D^0\phi(\xi(t,\bsy),\bsy)]^{r+q}},
\\[4mm]
 \mbox{with $r\in\bbN_0$, $\bsalpha=(\bsalpha_\ell)_{\ell=1}^r$,
 $\bsalpha_\ell\in \N_0^{d + 1}\!\setminus\!\{\bse_0,\bszero\}$, $\beta \in \N_0$ satisfying}
 \\[1mm]
 r\le 2|\bseta|+q-1,\;
 \alpha_{\ell,0} \le |\bseta|+q, \;
 \beta \leq |\bseta|+q-1,\;
 \beta\bse_0 + \displaystyle\sum_{\ell = 1}^{r} \bsalpha_\ell
 = (r+q-1, \bseta).
 \end{cases}
 \hspace{-0.8cm}
\end{align}
We assume that all such functions $h_{q,\bseta}$ satisfy
\begin{equation} \label{eq:h-lim}
 \lim_{\bsy \to \partial U_t} h_{q,\bseta}(t, \bsy) = 0
 \quad \text{for all } t \in [a, b],
\end{equation}
and there is a constant $B_{q,\bseta}$ such that
\begin{equation} \label{eq:h-int}
 \sup_{t \in [a, b]}
 \int_{U_t} |h_{q,\bseta}(t, \bsy)|^2 \,
 \bspsi_{\bseta}(\bsy_\bseta)\,\bsrho_{-\bseta}(\bsy_{-\bseta})\,\rd \bsy
 \le B_{q,\bseta} \,<\, \infty.
\end{equation}
\end{assumption}

It is worthwhile to briefly discuss Assumption~\ref{asm:phi-h}. As we
will see in the theorem below, the functions $h_{q, \bseta}$ occur when
differentiating a preintegrated function using the multivariate chain and
product rules. Loosely speaking, the parameter $q$ relates to
differentiating with respect to $t$, whereas $\bseta$ relates to
differentiating with respect to $\bsy$. The conditions \eqref{eq:h-lim}
and \eqref{eq:h-int} then ensure that the derivatives are well-behaved
enough for the preintegrated function to be sufficiently smooth.
It follows from Assumption~\ref{asm:phi} and Theorem~\ref{thm:implicit} that each
function $h_{q, \bseta}$ of the form \eqref{eq:h-form} is continuous on
$\overline{V}$.
For an appropriate choice of $\{\psi_i\}$, Assumption~\ref{asm:phi-h}
will hold for functions $\phi$ for which the preintegrated function
$P_0 g_t$ as in \eqref{eq:P0g-ind} (and also \eqref{eq:P0_dirac} below) is unbounded and 
has unbounded derivatives $h_{q, \bseta}$.
In this case, since $P_0 g_t$ is unbounded, after mapping
back to $[0, 1]^d$ it does not belong to BVHK.

Having introduced preintegration and our key assumptions, we now state
the main smoothing by preintegration theorem for functions of the form
\eqref{eq:g-ind}. It is a refined version of \cite[Theorem~3]{GKLS18}.

\begin{theorem} \label{thm:preint_ind}
Let $d\ge 1$, $\bsnu\in \bbN_0^d$, and $[a,b]\subset\bbR$. Suppose that
$\phi$ and $\rho_0$ satisfy Assumption~\ref{asm:phi} and
Assumption~\ref{asm:phi-h} for $q=0$ and all $\bszero \ne
\bseta\le\bsnu$. Then for $t \in [a, b]$, the function
\[
  g_t(y_0,\bsy) \coloneqq \ind(t-\phi(y_0, \bsy))
  \quad\mbox{satisfies}\quad
  P_0 g_t \in \calH^{\bsnu}_{d} \cap C^{\bsnu}(\R^d),
\]
with its $\calH^\bsnu_d$-norm bounded uniformly in $t$,
\begin{align} \label{eq:norm-cdf}
 \sup_{t \in [a, b]} \|P_0 g_t\|_{\calH^\bsnu_d}
 \le \Bigg(1 +
 \sum_{\bszero\ne\bseta \leq \bsnu}
 \frac{\big(8^{|\bseta|-1}(|\bseta|-1)!\big)^2 B_{0, \bseta}}{\gamma_\bseta}\Bigg)^{1/2}
 \,<\, \infty\,.
\end{align}
\end{theorem}

\begin{proof}
From \eqref{eq:P0g-ind} the preintegrated function can be written as
\begin{align*}
P_0 g_t(\bsy)
=
\begin{cases}
 \Phi_0(\xi(t,\bsy))
 & \text{if } \bsy \in U_t\,,\\
 0 & \text{if } \bsy \in \R^d \setminus U_t\,.
\end{cases}
\end{align*}
If $U_t = \emptyset$ then $P_0 g_t \equiv 0$ on $\R^d$, and the
result holds trivially. If $U_t \neq \emptyset$ then for any $\bseta
\in \N_0^d$ with $\bszero\ne\bseta \leq \bsnu$, we first prove by
induction on $|\bseta|\ge 1$ that the $\bseta$th derivative of $P_0 g_t$
is given by
\begin{align} \label{eq:D-P0-ind}
D^\bseta [P_0 g_t(\bsy)]
=
\begin{cases}
 \displaystyle
 \sum_{j=1}^{J_{0,\bseta}} h_{0,\bseta}^{[j]}(t,\bsy)
 & \text{if } \bsy \in U_t\,, \quad\mbox{with}\quad J_{0,\bseta} \le 8^{|\bseta|-1}(|\bseta|-1)!\,, \\
 0 & \text{if } \bsy \in \R^d \setminus U_t\,,
\end{cases}
\end{align}
where each function $h_{0,\bseta}^{[j]}$ is of the form \eqref{eq:h-form}
with $q=0$.

For the base case $\bseta = \bse_i$ with any $i\in \{1:d\}$, we take
$r=1$, $\bsalpha_1 = \bse_i$, $\beta = 0$ and $J_{0,\bse_i} = 1$ to
recover the single function \eqref{eq:first-der}. Suppose next that
\eqref{eq:D-P0-ind} holds for some $\bseta\in\bbN_0^d$ with
$|\bseta|\ge 1$ and consider any $i\in\{1:d\}$ and $\bsy\in U_t$. We have
\begin{align}
\label{eq:D^iD^eta-P0g}
  D^i D^\bseta [P_0 g_t(\bsy)]
 &= \sum_{j=1}^{J_{0,\bseta}} D^i h_{0,\bseta}^{[j]}(t,\bsy)
 = \sum_{j=1}^{J_{0,\bseta}} \sum_{k=1}^{K_{0,\bseta}} h_{0,\bseta+\bse_i}^{[j,k]}(t,\bsy)
\nonumber\\
 &= \sum_{j'=1}^{J_{0,\bseta+\bse_i}} h_{0,\bseta+\bse_i}^{[j']}(t,\bsy).
\end{align}
In the second equality we used Lemma~\ref{lem:count-y} in
Appendix~\ref{app:tech} with $q=0$, which states that each function
$D^i h_{0,\bseta}^{[j]}$ can be written as a sum of $K_{0,\bseta}\le
8|\bseta|-3$ functions of the form \eqref{eq:h-form} with $\bseta$
replaced by $\bseta+\bse_i$. We enumerated these functions with the
notation $h_{0,\bseta+\bse_i}^{[j,k]}$ and then relabeled all functions
for different combinations of indices $j$ and $k$ with the notation
$h_{0,\bseta+\bse_i}^{[j']}$. The total number of functions satisfies
\[
 J_{0,\bseta+\bse_i} = J_{0,\bseta}\, K_{0,\bseta}
 \le 8^{|\bseta|-1}(|\bseta|-1)!\, (8|\bseta|-3)
 \le 8^{|\bseta+\bse_i|-1}\,(|\bseta+\bse_i|-1)!\,,
\]
as required. This completes the induction proof for \eqref{eq:D-P0-ind}.

Since, for all $\bseta \leq \bsnu$, every function $h_{0,\bseta}^{[j]}(t,\cdot)$
in \eqref{eq:D-P0-ind} is continuous on $U_t$, it follows that $P_0 g_t \in
C^{\bsnu}(U_t)$. Also, $P_0 g_t\equiv 0$ on $\R^d \setminus U_t$ is
clearly smooth, and so we just need the derivatives to be continuous
across the boundary $\partial U_t$. Indeed, the assumption
\eqref{eq:h-lim} implies that
$D^\bseta [P_0 g_t(\bsy)] \to 0$ as $\bsy \to\partial U_t$ for all $\bseta \leq \bsnu$.
Hence, it follows by an adaptation of \cite[Lemma~9]{GKLS18} that $P_0
g_t \in C^{\bsnu}(\R^d)$.

It remains to show that $P_0 g_t \in \calH_{d}^{\bsnu}$ by estimating its
norm. We have
\begin{align*}
&\|P_0 g_t\|_{\calH_{d}^{\bsnu}}^2
 = \sum_{\bseta \leq \bsnu} \frac{1}{\gamma_\bseta}
 \int_{\R^d} \big|D^\bseta [P_0 g_t(\bsy)]\big|^2\,
 \bspsi_{\bseta}(\bsy_\bseta)\,\bsrho_{-\bseta}(\bsy_{-\bseta})\,\rd \bsy
 \\
&= \int_{U_t} |\Phi_0(\xi(t,\bsy))|^2\,\bsrho(\bsy) \,\rd \bsy
+ \sum_{\bszero\ne\bseta \leq \bsnu} \frac{1}{\gamma_\bseta}
\int_{U_t} \Bigg|
 \sum_{j=1}^{J_{0,\bseta}} h_{0,\bseta}^{[j]}(t,\bsy)\Bigg|^2
 \bspsi_{\bseta}(\bsy_\bseta)\,\bsrho_{-\bseta}(\bsy_{-\bseta})\,\rd \bsy
 \\
&\le 1
+ \sum_{\bszero\ne\bseta \leq \bsnu} \frac{J_{0,\bseta}}{\gamma_\bseta}
 \sum_{j=1}^{J_{0,\bseta}} \int_{U_t} | h_{0,\bseta}^{[j]}(t,\bsy)|^2\,
 \bspsi_{\bseta}(\bsy_\bseta)\,\bsrho_{-\bseta}(\bsy_{-\bseta})\,\rd \bsy
 \\
&\le 1
 + \sum_{\bszero\ne\bseta \leq \bsnu} \frac{\big(8^{|\bseta|-1}(|\bseta|-1)!\big)^2 B_{0,\bseta}}
 {\gamma_\bseta}
 \,<\, \infty\,,
\end{align*}
where we used the assumption \eqref{eq:h-int} with $q=0$. This
completes the proof.
\end{proof}

We remark that it would suffice to have $\nu_0 \coloneqq |\bsnu|$ and
$\rho_0\in C^{|\bsnu|-1}(\bbR)$ in Assumption~\ref{asm:phi} for
Theorem~\ref{thm:preint_ind} to hold. In other words, we have assumed an
extra order of regularity on $\phi$ and $\rho_0$ with respect to $y_0$
beyond that required for the cdf. The extra regularity is needed for the
corresponding theorem for the pdf, see Theorem~\ref{thm:preint_dirac} below.

The bound on the norm \eqref{eq:norm-cdf} with constants
$\{B_{0, \bseta}\}$ (e.g., obtained from information on specific
problems) can be used to choose the weight parameters $\{\gamma_\bseta\}$
to model the relative importance of subsets of variables. This would in
turn allow for a complete error analysis that is explicit in the
dependence on dimension. Performing this analysis for specific problems
will be pursued in future work.

\subsection{Smoothing by preintegration for Dirac $\delta$ distributions}

\label{sec:preint-delta}

In this section, we show that the same smoothing by preintegration
theory also works for distributions that are constructed by a Dirac
$\delta$ function, which will allow us to also estimate the pdf as
formulated in \eqref{eq:pdf-expect}.

For $t\in [a,b]$, consider a distribution of the form
\begin{equation*}
 g_t(y_0, \bsy) = \delta\big(t-\phi(y_0, \bsy)\big),
\end{equation*}
where $\delta(\cdot)$ is the Dirac $\delta$ function as characterised by
\eqref{eq:dirac} and $\phi : \R^{d + 1} \to \R$ satisfies
Assumption~\ref{asm:phi}.

For $t\in [a,b]$, $\bsy \in U_t$, and assuming $U_t\ne\emptyset$, we have
$t\in \phi(\bbR,\bsy)$.
Let $\xi(t, \bsy)$ be the unique point of
discontinuity in dimension~$0$ as in \eqref{eq:xi}. Applying the
preintegration operator \eqref{eq:P0} to the distribution $g_t$ and using
the change of variables $z = \phi(y_0, \bsy)$ so that $y_0 = \xi(z,
\bsy)$, we obtain
\begin{align}
\label{eq:P0_dirac}
P_0 g_t(\bsy) &=
\int_{-\infty}^\infty \delta\big(t-\phi(y_0, \bsy)\big)\, \rho_0(y_0) \, \rd y_0
= \int_{\phi(\R, \bsy)} \delta(t-z)\, \rho_0(\xi(z, \bsy))\,
\pd{}{}{z}\xi(z, \bsy)\, \rd z
\nonumber\\
&= \int_{\phi(\R, \bsy)} \delta(t-z)\, \rho_0(\xi(z, \bsy))\,
\frac{1}{D^0\phi(\xi(z, \bsy), \bsy)}\, \rd z
= \frac{\rho_0(\xi(t, \bsy))}{D^0\phi(\xi(t, \bsy), \bsy)},
\end{align}
where we used \eqref{eq:xi}, \eqref{eq:dt-xi}, and the definition of the
$\delta(\cdot)$ function \eqref{eq:dirac}. For $\bsy\in\bbR^d\setminus
U_t$ and so $t \not\in \phi(\R, \bsy)$, we have $\delta(t-z) = 0$ for
\emph{all} $z\in \phi(\R, \bsy)$, and hence $P_0 g_t(\bsy) = 0$. As
expected, \eqref{eq:P0_dirac} is the derivative of \eqref{eq:P0g-ind} with
respect to $t$.

With a similar proof to Theorem~\ref{thm:preint_ind}, we now show that the
preintegrated distribution $P_0 g_t$ is also smooth, with a different
bound on its norm.

\begin{theorem} \label{thm:preint_dirac}
Let $d\ge 1$, $\bsnu\in \bbN_0^d$, and $[a,b]\subset\bbR$. Suppose that
$\phi$ and $\rho_0$ satisfy Assumption~\ref{asm:phi} and
Assumption~\ref{asm:phi-h} for $q=1$ and all $\bseta\le\bsnu$. Then for
$t \in [a, b]$, the distribution
\[
  g_t(y_0,\bsy) \coloneqq \delta\big(t-\phi(y_0, \bsy)\big)
  \quad\mbox{satisfies}\quad
  P_0 g_t \in \calH^{\bsnu}_{d} \cap C^{\bsnu}(\R^d),
\]
with its $\calH^\bsnu_d$-norm bounded uniformly in $t$,
\begin{equation} \label{eq:norm-pdf}
 \sup_{t \in [a, b]} \|P_0 g_t\|_{\calH^\bsnu_d}
 \le
 \Bigg(\sum_{\bseta \leq \bsnu} \frac{\big(8^{|\bseta|}|\bseta|!\big)^2 B_{1,\bseta}}{\gamma_\bseta}\Bigg)^{1/2}
 \,<\, \infty\,.
\end{equation}
\end{theorem}

\begin{proof}
From \eqref{eq:P0_dirac} the preintegrated distribution can be written as
\begin{align*}
P_0 g_t(\bsy)
=
\begin{cases}
\displaystyle
\frac{\rho_0(\xi(t, \bsy)}{D^0 \phi(\xi(t, \bsy), \bsy)}
& \text{if } \bsy \in U_t\,,\\
0 & \text{if } \bsy \in \R^d \setminus U_t\,.
\end{cases}
\end{align*}

The proof follows the same strategy the proof of Theorem~\ref{thm:preint_ind}, but
now with $q = 1$.
Again, $P_0 g_t \equiv 0$ for the case $U_t = \emptyset$, so the result holds trivially.

For $\bseta \leq \bsnu$, we first prove by
induction on $|\bseta|$ that the $\bseta$th derivative of $P_0 g_t$ is
given by \eqref{eq:D-P0-ind} with $q = 1$ (instead of 0),
where each $h^{[j]}_{1, \bseta}$ is of the form \eqref{eq:h-form}
with $q = 1$ and $J_{1, \bseta} \leq 8^{|\bseta|}|\bseta|!$.
For the base case $\bseta = \bszero$, we take $r=0$ (and so there are no
$\bsalpha_\ell$ terms), $\beta = 0$, and $J_{1,\bszero} = 1$ to recover
the single function $\rho_0(\xi(t, \bsy))/D^0\phi(\xi(t, \bsy), \bsy)$.
For the inductive step, Lemma~\ref{lem:count-y} with $q = 1$
implies that \eqref{eq:D^iD^eta-P0g} holds with $q = 1$.
Since in this case $K_{1, \bseta} \leq 8|\bseta| - 3$,
the total number of functions satisfies
\[
 J_{1,\bseta+\bse_i} = J_{1,\bseta}\, K_{1,\bseta}
 \le 8^{|\bseta|}|\bseta|!\, (8|\bseta|+3)
 \le 8^{|\bseta+\bse_i|}\,|\bseta+\bse_i|!\,,
\]
as required. This completes the induction proof for \eqref{eq:D-P0-ind} with $q = 1$.

Since each $h_{1,\bseta}^{[j]}(t,\cdot) \in C(U_t)$ and by
\eqref{eq:h-lim}, it follows that $P_0g_t \in C^\bsnu(U_t)$
with $D^\bseta P_0 g_t(\bsy) \to 0$ as $\bsy \to \partial U_t$
for all $\bseta \leq \bsnu$. Hence,
\cite[Lemma~9]{GKLS18} again implies that $P_0
g_t \in C^{\bsnu}(\R^d)$.

Finally, it remains to show that $P_0 g_t \in \calH_{d}^{\bsnu}$. The norm
of $P_0 g_t$ is given by
\begin{align*}
&\|P_0 g_t\|_{\calH_{d}^{\bsnu}}^2
=
\sum_{\bseta \leq \bsnu} \frac{1}{\gamma_\bseta}
\int_{U_t} \Bigg|
 \sum_{j=1}^{J_{1,\bseta}} h_{1,\bseta}^{[j]}(\bsy)\Bigg|^2\,
 \bspsi_{\bseta}(\bsy_\bseta)\,\bsrho_{-\bseta}(\bsy_{-\bseta})\,\rd \bsy
\\
&\le
\sum_{\bseta \leq \bsnu} \frac{J_{1,\bseta}}{\gamma_\bseta}
 \sum_{j=1}^{J_{1,\bseta}} \int_{U_t} | h_{1,\bseta}^{[j]}(\bsy)|^2\,
 \bspsi_{\bseta}(\bsy_\bseta)\,\bsrho_{-\bseta}(\bsy_{-\bseta})\,\rd \bsy
 \le
 \sum_{\bseta \leq \bsnu} \frac{\big(8^{|\bseta|}|\bseta|!\big)^2 B_{1,\bseta}}
 {\gamma_\bseta}
 \,<\, \infty\,,
\end{align*}
where we used the assumption \eqref{eq:h-int} with $q=1$. This
completes the proof.
\end{proof}

\section{Distribution function and density estimators}
\label{sec:algo}

In this section we briefly outline the QMC with preintegration algorithms for
approximating the cdf and pdf. First, note that the cdf and pdf can be
written as $d$-dimensional integrals after carrying out the preintegration
step. Explicitly, by Fubini's Theorem we can write the representation
\eqref{eq:cdf-expect} for the cdf as
\begin{align}
\label{eq:cdf_simple}
F(t) = \int_{\R^d} P_0\big(\ind\big(t-\phi(\cdot,\bsy)\big)\big) \, \bsrho(\bsy)
\, \rd \bsy
= \int_{U_t}\Phi_0(\xi(t, \bsy))\bsrho(\bsy)
\, \rd \bsy\,,
\end{align}
where in the last step we have substituted in the simplified formula
\eqref{eq:P0g-ind} for a preintegrated indicator function. Similarly,
using the representation \eqref{eq:pdf-expect} along with
\eqref{eq:P0_dirac} we can write
\begin{align}
\label{eq:pdf_simple}
f(t) = \int_{\R^d} P_0\big(\delta\big(t-\phi(\bsy)\big)\big) \bsrho(\bsy) \, \rd \bsy
=
\int_{U_t} \frac{\rho_0(\xi(t, \bsy))}{D^0\phi(\xi(t, \bsy), \bsy)}
\bsrho(\bsy)\,\rd\bsy\,.
\end{align}

\subsection{Pointwise approximation}

As a start, we consider approximating $F$ and $f$ pointwise at $t \in
[a, b]$. Applying an $N$-point QMC rule \eqref{eq:qmc} to the
$d$-dimensional integrals \eqref{eq:cdf_simple} and \eqref{eq:pdf_simple},
we obtain the approximations $\Fhat_N$ and $\fhat_N$ as follows:
\begin{align}
\label{eq:F_N_point}
 F(t) \approx
 \Fhat_N(t)
 &\coloneqq Q_{d, N}\big(\Phi_0 (\xi(t, \cdot))\big)
=
\frac{1}{N} \sum_{n=0}^{N - 1} \Phi_0\big(\xi(t, \bstau_n)\big),
\\
\label{eq:f_N_point}
 f(t) \approx
 \fhat_N(t)
 &\coloneqq Q_{d, N}\bigg(\frac{\rho_0 (\xi(t, \cdot))}{D^0\phi (\xi(t, \cdot), \cdot)}\bigg)
= \frac{1}{N} \sum_{n=0}^{N - 1} \frac{\rho_0(\xi(t, \bstau_n))}{D^0\phi(\xi(t, \bstau_n), \bstau_n)}\,.
\end{align}
We are particularly interested in using randomly shifted lattice points
\eqref{eq:lattice} with \eqref{eq:transQMC}, but the description in this
section applies equally to other QMC rules.

Since a randomly shifted lattice rule $Q_{d, N}$ is an unbiased estimator
of the $d$-dimensional integral, it follows that the estimators $\Fhat_N$
and $\fhat_N$ are also \emph{unbiased}. However, this assumes that we can
compute the point of discontinuity $\xi(t, \cdot)$ exactly, which is not
generally true. In practice, we must often approximate this $\xi(t, \cdot)$
by some numerical approximation, as performed in
\cite{GKLS18,LEcPuchBAb19}
for the case where $\bsY$ is a multivariate normal vector.
This leads to \emph{biased} estimators
$\Ftilde_N$ and $\ftilde_N$.
We now detail how to construct these biased
estimators efficiently.

First, consider approximating the cdf $F$ at $t \in [a, b]$. For each transformed QMC
point $\bstau_n$ we must: (i) find the point of discontinuity $\xi(t,
\bstau_n)$, and then (ii) evaluate $\Phi_0$ at this point. In practice,
these two actions must be performed numerically, however, we stress that
we only need to work with the univariate function
\[
  \phi_{0,n} \coloneqq \phi(\cdot,\bstau_n),
\]
which can be evaluated efficiently for multiple inputs if we
``precompute'' and store the contribution of $\bstau_n$ to $\phi_{0,n}$.
As a trivial example to demonstrate this, if we have a product function
$\phi(\cdot,\bstau_n) = p_0(\cdot)\,\prod_{i=1}^d p_i(\tau_{n,i})$, then
evaluating $\phi$ in general has a cost of $\calO(d)$, but if we
precompute and store the product involving $\bstau_n$, then we can evaluate
$\phi_{0,n}$ for $K$ different inputs with a cost of $\calO(K+d)$ instead
of $\calO(Kd)$. We assume that $\phi_{0,n}' = D^0\phi(\cdot, \bstau_n)$
can be evaluated directly, and we also precompute and store the
contribution of $\bstau_n$ to $\phi_{0,n}'$.

To find the point of discontinuity, we use a numerical root-finding
algorithm, e.g., Newton's method. Since $\phi \in C^{(\nu_0,\bsnu)}(\R^{d
+ 1})$ with $\nu_0 = |\bsnu|+1$, we have $\phi_{0,n}\in C^{\nu_0}(\bbR)$
for each $\bstau_n$. If $|\bsnu|\ge 1$ then $\phi_{0,n} \in C^2(\bbR)$ and
Newton's method converges quadratically, so in practice only a few
iterations are required. Alternatively, if the higher-order derivatives of
$\phi_{0,n}$ can be computed explicitly, then a higher-order Householder
method can instead be used. We denote the numerical approximation of $\xi$
by $\widetilde{\xi}$.

If $\rho_0$ is a Gaussian distribution, then fast and accurate
approximations of its cdf $\Phi_0$ are readily available.
Otherwise, if we cannot evaluate $\Phi_0$ easily, then we
approximate the one-dimensional integral $\Phi_0(y_0) =
\int_{-\infty}^{y_0} \rho_0(z)\,\rd z$ by a quadrature rule. In both
cases we denote the approximation of $\Phi_0$ by
$\widetilde{\Phi}_0$.

Approximating the pdf $f$ at $t\in [a,b]$ is similar: (i) obtaining the
point $\xi(t,\bstau_n)$ is the same, while (ii) evaluating the ratio
$\rho_0/\phi_{0,n}'$ instead of the one-dimensional integral for $\Phi_0$
is slightly simpler.

The QMC with approximate preintegration estimators of the cdf $F$ and pdf
$f$ are
\begin{align}
\label{eq:F_point_approx}
\Ftilde_N(t) &\coloneqq \frac{1}{N} \sum_{n=0}^{N - 1}
\widetilde{\Phi}_0 \big(\widetilde{\xi}(t, \bstau_n)\big)\,,
 \\
\label{eq:f_point_approx}
\ftilde_N(t) &\coloneqq \frac{1}{N} \sum_{n=0}^{N - 1}
\frac{\rho_0 \big(\widetilde{\xi}(t, \bstau_n)\big)}{D^0\phi\big(\widetilde{\xi}(t, \bstau_n), \bstau_n\big)} \,.
\end{align}
Algorithms \ref{alg:cdf-pointwise} and~\ref{alg:pdf-pointwise} give
explicit implementations of \eqref{eq:F_point_approx} and
\eqref{eq:f_point_approx}.

\subsection{Cost of pointwise approximation}
\label{sec:cost_point}

First, in the special case where the point of discontinuity $\xi(t,
\cdot)$ and the one-dimensional integral $\Phi_0$ can be computed
analytically, we have $\cost(\Fhat_N(t)) = \bigO(N)$ and
$\cost(\fhat_N(t)) = \bigO(N)$. However, as mentioned above, this is not
the typical case in practice, and we must approximate these quantities by
numerical root-finding and quadrature methods.

\begin{algorithm}[t]
\caption{Pointwise cdf estimator} \label{alg:cdf-pointwise} Given $t \in
[a, b]$, $N \in \N$ and $\{\bstau_n\}_{n=0}^{N - 1}$ a transformed $d$-dimensional QMC
point set:
\begin{algorithmic}[1]
 \State Initialise: $\Ftilde_N(t) \leftarrow 0$
 \For{$n = 0, 1, \ldots, N - 1$}
 \State Precompute the contribution of $\bstau_n$ to $\phi_{0,n} = \phi(\cdot,
 \bstau_n)$ and $\phi_{0,n}' = D^0\phi(\cdot,\bstau_n)$
 \State Compute the point of discontinuity $\widetilde{\xi}(t, \bstau_n)$
 \State Approximate the 1D integral
$\widetilde{\Phi}_0 \big(\widetilde{\xi}(t, \bstau_n)\big)
 \approx 
 \int_{-\infty}^{\widetilde{\xi}(t, \bstau_n) } \rho_0(z) \, \rd z
$
 \State Sum: $\Ftilde_N(t) \leftarrow \Ftilde_N(t)+ \widetilde{\Phi}_0 \big(\widetilde{\xi}(t, \bstau_n)\big)$
 \EndFor
 \State Average: $\Ftilde_N(t) \leftarrow \Ftilde_N(t)/N$
\end{algorithmic}
\end{algorithm}
\begin{algorithm}[t]
\caption{Pointwise pdf estimator} \label{alg:pdf-pointwise} Given $t \in
[a, b]$, $N \in \N$ and $\{\bstau_n\}_{n=0}^{N - 1}$ a transformed $d$-dimensional QMC
point set:
\begin{algorithmic}[1]
 \State Initialise: $\ftilde_N(t) \leftarrow 0$
 \For{$n = 0, 1, \ldots, N - 1$}
 \State Precompute the contribution of $\bstau_n$ to $\phi_{0,n} = \phi(\cdot,
 \bstau_n)$ and $\phi_{0,n}' = D^0\phi(\cdot,\bstau_n)$
 \State Compute the point of discontinuity $\widetilde{\xi}(t, \bstau_n)$
 \State Sum: $
 \ftilde_N(t) \leftarrow \ftilde_N(t)+ \frac{\rho_0 \big(\widetilde{\xi}(t,
 \bstau_n)\big)}{D^0\phi\big(\widetilde{\xi}(t, \bstau_n), \bstau_n\big)}$
 \EndFor
 \State Average: $\ftilde_N(t) \leftarrow \ftilde_N(t)/N$
\end{algorithmic}
\end{algorithm}

To analyse the cost of the pointwise approximations $\Ftilde_N(t)$ and
$\ftilde_N(t)$ in Algorithms~\ref{alg:cdf-pointwise} and
\ref{alg:pdf-pointwise}, we assume that the number of evaluations of
$\phi_{0,n}$ and $\phi_{0,n}'$ in the root-finding method to compute
$\widetilde{\xi}(t, \bstau_n)$ for each $n$ in Step~4 is bounded by
$K_\mathrm{root}$, which is assumed to be independent of $n$. For the cdf
approximation in Algorithm~\ref{alg:cdf-pointwise}, we also assume that
$\cost(\rho_0) = \bigO(1)$, and that the number of quadrature points to
compute the one-dimensional integral in Step~5 is bounded by
$K_\mathrm{quad}$, also independently of~$n$.

Then to more concretely illustrate why it is important to precompute the
contribution of $\bstau_n$ to $\phi_{0,n}$ and $\phi_{0,n}'$, we make the
following assumption about the difference in cost in evaluating $\phi$ and
$D^0\phi$ compared with the univariate functions $\phi_{0,n}$ and
$\phi_{0,n}'$:
\begin{align}
\label{eq:dollar}
\begin{cases}
 \cost(\phi) = \$(d),
 &\!\cost(\phi_{0,n}) = \$(1) \text{ with precomputed contribution of $\bstau_n$}, \\
 \cost(D^0\phi) = \$(d),
 &\!\cost(\phi_{0,n}') = \$(1) \text{ with precomputed contribution of $\bstau_n$},
 \end{cases}
\end{align}
for some nondecreasing function $\$ : \N \to \N$.

The cost of Algorithms~\ref{alg:cdf-pointwise} and \ref{alg:pdf-pointwise}
are then
\begin{align*}
 \cost\big(\Ftilde_N(t)\big) &= \bigO\big(N\,[\$(d) + K_\mathrm{root}\,\$(1) + K_\mathrm{quad}]\big)
 \quad \text{and} \\
 \cost\big(\ftilde_N(t)\big) &= \bigO\big(N\,[\$(d) + K_\mathrm{root}\,\$(1)]\big)\,.
\end{align*}
For large $d$, this will be much more efficient than a naive implementation
without precomputed contribution of $\bstau_n$, which would have
\begin{align*}
\cost\big(\Ftilde_N(t)\big) = \bigO\big(N\,[K_\mathrm{root}\,\$(d) + K_\mathrm{quad}]\big)
\quad \text{and} \quad
\cost\big(\ftilde_N(t)\big) = \bigO\big(N\,K_\mathrm{root}\,\$(d) \big)\,.
\end{align*}

\subsection{Approximating the cdf and pdf on an interval}

Now we outline the full QMC with preintegration method for approximating
the cdf and pdf on $[a, b]$, obtained by applying Lagrange interpolation
$L_M$ based on points {$\{t_m\}_{m=0}^M \subset [a, b]$ to the pointwise
estimators $\Fhat_N$ and $\fhat_N$. We denote the approximations~by
\begin{align}
\label{eq:F_NM}
 \Fhat_{N, M}&\coloneqq L_M(\Fhat_N) =
 L_M\Big(Q_{d, N}\big(\Phi_0(\xi(\bullet,\cdot))\big)\Big),
 \\
 \label{eq:f_NM}
 \fhat_{N, M}&\coloneqq L_M(\fhat_N) =
 L_M\Big(Q_{d, N}\Big( \frac{\rho_0(\xi(\bullet,\cdot))}{D^0\phi(\xi(\bullet,\cdot),\cdot)}\Big)\Big),
\end{align}
where the QMC rule $Q_{d, N}$ acts on a function with respect to $\cdot$
whereas Lagrange interpolation $L_M$ acts on $\bullet$. As discussed in
Section~\ref{sec:1D-approx} we will use Chebyshev points, but the
description below allows for any set of distinct interpolation nodes.

In practice, we must approximate the point of discontinuity by
$\widetilde{\xi} \approx \xi$, and for the cdf also the one-dimensional
integral by $\widetilde{\Phi}_0 \approx \Phi_0$. This leads to the biased
estimators
\begin{align}
\label{eq:F_NM_approx}
\Ftilde_{N, M}&\coloneqq L_M(\Ftilde_N)
= L_M\Big(Q_{d, N}\Big(\widetilde{\Phi}_0\big(\widetilde{\xi}(\bullet,\cdot)\big)\Big)\Big),
\\
\label{eq:f_NM_approx}
\ftilde_{N, M}&\coloneqq L_M(\ftilde_N) =
 L_M\bigg(Q_{d, N}\bigg(
 \frac{\rho_0\big(\widetilde{\xi}(\bullet,\cdot)\big)}{D^0\phi\big(\widetilde{\xi}(\bullet,\cdot), \cdot\big)}
 \bigg)\bigg).
\end{align}

Recall from the definition of the Lagrange interpolation operator
\eqref{eq:L_M} that to construct the estimators $\Ftilde_{N, M}$ and
$\ftilde_{N, M}$, we must compute the pointwise approximations $\Ftilde_N$
and $\ftilde_N$ at all of the interpolation nodes $\{t_m\}_{m = 0}^M$. One
way to implement the estimator $\Ftilde_{N, M}$ as in
\eqref{eq:F_NM_approx}  is to simply run Algorithm~\ref{alg:cdf-pointwise}
for each $t_m$ for $m = 0, 1, \ldots, M$, with $\cost(\Ftilde_{N, M}) = (M
+ 1)\times \cost(\Ftilde_N)$. However, since we can use the same QMC rule
for each interpolation node $t_m$, it is more efficient to instead
\emph{vectorise} Algorithm~\ref{alg:cdf-pointwise} and utilise precomputed
contributions of each point $\bstau_n$ so that we only have to deal
with $M + 1$ univariate functions. Similar arguments can also be made for
the cdf estimator $\ftilde_{N, M}$. Explicit algorithms detailing how to
construct the estimators $\Ftilde_{N, M}$ and $\ftilde_{N, M}$ are given
in Algorithms~\ref{alg:cdf-full} and~\ref{alg:pdf-full}.

\subsection{Cost of full cdf and pdf estimators}
\label{sec:cost_full}

Following the analysis of the cost of the pointwise estimators in
Section~\ref{sec:cost_point}, we again assume the cost model
\eqref{eq:dollar} and assume that the number of evaluations of
the univariate functions in the root-finding method and the number
of quadrature points for computing the one-dimensional
integral are bounded by $K_\mathrm{root}$ and $K_\mathrm{quad}$,
respectively, which are additionally assumed to be independent of $n$
and $m$. The cost of Algorithms~\ref{alg:cdf-full} and \ref{alg:pdf-full}
are then
\begin{align*}
\cost(\Ftilde_{N, M}) &= \bigO\big(N\,[\$(d) + M\,K_\mathrm{root}\,\$(1) + M\,K_\mathrm{quad}]\big),\\
\cost(\ftilde_{N, M}) &= \bigO\big(N\,[\$(d) + M\,K_\mathrm{root}\,\$(1)]\big).
\end{align*}

To once again illustrate the importance of the precomputation step, we note
that a na\"ive implementation that simply evaluates $\phi$ at all of its
components each time would have $\cost(\ftilde_{N, M}) =
\bigO(N\,M\,K_\mathrm{root}\,\$(d))$.

\begin{algorithm}[t]
\caption{cdf estimator} \label{alg:cdf-full} Given $M\in\bbN$, $\{t_m\}_{m
= 0}^M \subset [a, b]$, $N \in \N$ and $\{\bstau_n\}_{n=0}^{N - 1}$ a transformed
$d$-dimensional QMC point set:
\begin{algorithmic}[1]
 \State Initialise: $\Ftilde_N(t_m) \leftarrow 0$ for each $m = 0,1,\ldots,M$
 \For{$n = 0, 1, \ldots, N - 1$}
 \State Precompute the contribution of $\bstau_n$ to $\phi_{0,n} = \phi(\cdot,
 \bstau_n)$ and $\phi_{0,n}' = D^0\phi(\cdot,\bstau_n)$
 \For{$m = 0, 1, \ldots, M$}
 \State Compute the point of discontinuity $\widetilde{\xi}(t_m, \bstau_n)$
 \State Approximate the 1D integral
$\widetilde{\Phi}_0 \big(\widetilde{\xi}(t_m, \bstau_n)\big) \approx
 \int_{-\infty}^{\widetilde{\xi}(t_m, \bstau_n) } \rho_0(z) \, \rd z$
 \State Sum: $\Ftilde_N(t_m) \leftarrow \Ftilde_N(t_m)
 + \widetilde{\Phi}_0\big(\widetilde{\xi}(t_m, \bstau_n)\big)$
 \EndFor
 \EndFor
 \State Average: $\Ftilde_N(t_m) \leftarrow \Ftilde_N(t_m)/N$ for each $m = 0, 1, \ldots, M$
 \State Interpolate: $
 \Ftilde_{N, M} \,\leftarrow\,
\sum_{m = 0}^M \Ftilde_N(t_m)\, \chi_{M, m}$
\end{algorithmic}
\end{algorithm}

\begin{algorithm}[t] \caption{pdf estimator} \label{alg:pdf-full}
Given $M\in\bbN$, $\{t_m\}_{m = 0}^M \subset [a, b]$, $N \in \N$ and
$\{\bstau_n\}_{n=0}^{N - 1}$ a transformed $d$-dimensional QMC point set:
\begin{algorithmic}[1]
 \State Initialise: $\ftilde_N(t_m) \leftarrow 0$ for each $m = 0,1,\ldots,M$
 \For{$n = 0, 1, \ldots, N - 1$}
 \State Precompute the contribution of $\bstau_n$ to $\phi_{0,n} = \phi(\cdot,
 \bstau_n)$ and $\phi_{0,n}' = D^0\phi(\cdot,\bstau_n)$
 \For{$m = 0, 1, \ldots, M$}
 \State Compute the point of discontinuity $\widetilde{\xi}(t_m, \bstau_n)$
 \State Sum: $
 \ftilde_N(t_m) \leftarrow \ftilde_N(t_m)+
 \frac{\rho_0 \big(\widetilde{\xi}(t_m,\bstau_n)\big)}{D^0\phi\big(\widetilde{\xi}(t_m, \bstau_n), \bstau_n\big)}$
 \EndFor
 \EndFor
 \State Average: $\ftilde_N(t_m) \leftarrow \ftilde_N(t_m)/N$ for each $m = 0, 1, \ldots, M$
 \State Interpolate: $
 \ftilde_{N, M} \,\leftarrow\, \sum_{m = 0}^M \ftilde_N(t_m)\, \chi_{M, m}$
\end{algorithmic}
\end{algorithm}

\section{Error analysis}

\subsection{Regularity of $F$ and $f$}
\label{sec:regularity}

In order to utilise the results on the error for interpolation on $[a, b]$
from Section~\ref{sec:1D-approx}, we need to know quantitatively how smooth
the cdf $F$ and pdf $f$ are with respect to $t$. Clearly this smoothness
will depend on the smoothness of the original transformation $\phi$ from
\eqref{eq:X}. Since in Assumption~\ref{asm:phi} we assume that $\phi$ is
$|\bsnu| + 1$ times differentiable with respect to variable $y_0$
and $\rho_0 \in C^{|\bsnu|}(\R)$,
we can expect a similar level of smoothness for $F$ and $f$.

To see the dependence on $t$ more explicitly, recall that the formulas
\eqref{eq:cdf-expect} for cdf and \eqref{eq:pdf-expect} for the pdf can be
formulated as $d$-dimensional integrals as in \eqref{eq:cdf_simple} and
\eqref{eq:pdf_simple}, respectively. From these formulas, it is then clear
that the smoothness of $F$ and $f$ depends on the smoothness of $\xi$,
which in turn depends on the smoothness of $\phi$. In particular,
Theorem~\ref{thm:implicit} implies that $\xi$ is as smooth (with respect to
$t$) as $\phi$ (with respect to $y_0$).

Note that the assumptions we make here on the smoothness of $\phi$ and
$\rho_0$ are the same as those required for the preintegration step, i.e.,
we do not need any further smoothness assumptions.

\begin{theorem} \label{thm:f-smooth}
Let $d\ge 1$, $\bsnu \in\bbN_0^d$, and $[a,b]\subset\bbR$. Suppose that
$\phi$ and~$\rho_0$ satisfy Assumption~\ref{asm:phi} and
Assumption~\ref{asm:phi-h} for $\bseta = \bszero$ and all $q = 1, 2,
\ldots, |\bsnu|+1$.
Assume additionally that $U_t = \R^d$ for all $t \in [a, b]$. Then
  $F \in W^{|\bsnu|+1, \infty}[a,b]$ and
  $f \in W^{|\bsnu|, \infty}[a,b]$,
and for $q = 1, \ldots, |\bsnu| + 1$ the
derivatives are bounded by
\begin{equation*}
  \|F^{(q)}\|_{L^\infty}
  = \|f^{(q - 1)}\|_{L^\infty}
  \le 3^{q - 1}\,(q - 1)!\,B_{q,\bszero}^{1/2}.
\end{equation*}
\end{theorem}
\begin{proof}
We prove that the cdf satisfies $F \in W^{|\bsnu|+1, \infty}[a, b]$. Then
since $f = F'$, the result for the pdf follows immediately.

Consider the derivative of order $q \leq |\bsnu|+1$. First,
differentiating \eqref{eq:cdf_simple} with respect to~$t$ and
applying the Leibniz rule \cite[Theorem~4]{GKLS18} $q$
times, we obtain
\begin{equation}
\label{eq:F^q}
 F^{(q)}(t) =
\int_{\bbR^d} \pd{q}{}{t^q} \Phi_0(\xi(t, \bsy))\,\bsrho(\bsy) \, \rd \bsy\,.
\end{equation}

Recall that we have
\begin{equation}
\label{eq:dt_Phi0}
 \frac{\partial}{\partial t} \Phi_0(\xi(t,\bsy))
 = \frac{\rho_0(\xi(t,\bsy))}{D^0\phi(\xi(t,\bsy),\bsy)}.
\end{equation}
We now prove by induction on $q\ge 1$ that
\begin{align} \label{eq:hyp3}
 \frac{\partial^q}{\partial t^q} \Phi_0(\xi(t,\bsy))
 = \sum_{j=1}^{J_{q,\bszero}} h_{q,\bszero}^{[j]}(t,\bsy),
 \quad\mbox{with}\quad J_{q,\bszero} \le 3^{q-1}\,(q-1)!\,,
\end{align}
where each function $h_{q,\bszero}^{[j]}$ is of the form \eqref{eq:h-form}
with $\bseta=\bszero$. The base step $q=1$ holds for the single function
\eqref{eq:dt_Phi0} with $r=0$ (no $\bsalpha$), $\beta=0$, and
$J_{1,\bszero} = 1$. Suppose next that \eqref{eq:hyp3} holds for some
$q\ge 1$. Then we have
\begin{align*}
 \frac{\partial}{\partial t}  \bigg(\frac{\partial^q}{\partial t^q} \Phi_0(\xi(t,\bsy)) \bigg)
 = \sum_{j=1}^{J_{q,\bszero}} \frac{\partial}{\partial t} h_{q,\bszero}^{[j]}(t,\bsy)
 = \sum_{j=1}^{J_{q,\bszero}} \sum_{k=1}^{K_{q,\bszero}} h_{q+1,\bszero}^{[j,k]}(t,\bsy)
 = \sum_{j'=1}^{J_{q+1,\bszero}} h_{q+1,\bszero}^{[j']}(t,\bsy).
\end{align*}
In the second equality we used Lemma~\ref{lem:count-t} in Appendix~\ref{app:tech}
with $\bseta = \bszero$, which states that each function
$\frac{\partial}{\partial t} h_{q,\bszero}^{[j]}$ can be written as a sum
of $K_{q,\bszero}\le 3q-1$ functions of the form \eqref{eq:h-form}, with
$q$ replaced by $q+1$. We enumerated these functions with the notation
$h_{q+1,\bszero}^{[j,k]}$ and then relabeled all functions for different
combinations of indices $j$ and $k$ with the
notation~$h_{q+1,\bszero}^{[j']}$. The total number of functions satisfies
\[
 J_{q+1,\bszero} = J_{q,\bszero}\, K_{q,\bszero}
 \le 3^{q-1} (q-1)!\, (3q-1)
 \le 3^{q}\,q!\,,
\]
as required. This completes the induction proof for \eqref{eq:hyp3}.

Thus we have
\[
 f^{(q-1)}(t) = F^{(q)}(t)
 = \sum_{j=1}^{J_{q,\bszero}}
 \int_{\bbR^d} h_{q,\bszero}^{[j]}(t,\bsy)\, \bsrho(\bsy) \,\rd\bsy
\]
and
\begin{align*}
 \|f^{(q-1)}\|_{L^\infty} = \|F^{(q)}\|_{L^\infty}
 \le \sum_{j=1}^{J_{q,\bszero}} \sup_{t\in [a,b]}
 \int_{\bbR^d} \big|h_{q,\bszero}^{[j]}(t,\bsy)\big|\, \bsrho(\bsy) \,\rd\bsy
 \le 3^{q-1}\,(q-1)!\,B_{q,\bszero}^{1/2}, 
\end{align*}
where we used assumption \eqref{eq:h-int} with $\bseta = \bszero$ and the
Cauchy-Schwarz inequality. Hence $F \in W^{|\bsnu|+1, \infty}[a, b]$ and
also $f = F' \in W^{|\bsnu|, \infty}[a, b]$.
\end{proof}

\begin{remark}
In Theorem~\ref{thm:f-smooth}, we have assumed for simplicity that $U_t = \R$ 
for all $t \in [a, b]$, which implies that for each $\bsy \in \R^d$ and $t \in [a, b]$
there is some $y_0 \in \R$ such that $\phi(y_0, \bsy) = t$.
This can be viewed as a restriction on the interval $[a, b]$.
\end{remark}

\subsection{Error of cdf and pdf estimators}

In this subsection we analyse the error of the unbiased estimators from
Section~\ref{sec:algo}. First, we prove bounds for the
RMSE of the pointwise estimators $\Fhat_N$
and $\fhat_N$. Then we bound the \emph{root-mean-integrated-square error
(RMISE)} of the full estimators $\Fhat_{N, M}$ and $\fhat_{N, M}$ on
$[a,b]$. Recall that the expectation in the RMSE and RMISE, which we
denote by $\bbE_\bsDelta$, is taken with respect to the random shift
$\bsDelta$ in the lattice rule.
In this section, we assume that the generating vector is
constructed using the CBC algorithm from \cite{NK14}.

\begin{theorem}[Pointwise RMSE] \label{thm:err-point}
Let $d \geq 1$, $\bsnu = \bsone \in \N^d$, and $[a, b] \subset \R$. Suppose
that $\phi$ and $\rho_0$ satisfy Assumption~\ref{asm:phi} and
Assumption~\ref{asm:phi-h} for all $\bseta \in \{0, 1\}^d$ with $q = 0$
for the cdf case and $q = 1$ for the pdf case. Let $Q_{d, N}$ be a
CBC-constructed randomly shifted lattice rule as in \eqref{eq:qmc}.
Then, for $t \in [a, b]$, the estimators $\Fhat_N(t)$ and $\fhat_N(t)$ as given
in \eqref{eq:F_N_point} and \eqref{eq:f_N_point} satisfy, for all $\lambda
\in (1/\omega, 1]$,
\begin{align}
\label{eq:err-point-F}
\sqrt{\bbE_\bsDelta \big[\big|F(t) - \Fhat_N(t)\big|^2\big]} \le C_{F,\lambda}\,
\phi_\mathrm{tot}(N)^{-1/(2\lambda)}\,,
\\
\label{eq:err-point-f}
\sqrt{\bbE_\bsDelta \big[\big|f(t) - \fhat_N(t)\big|^2\big]} \le C_{f,\lambda}\,
\phi_\mathrm{tot}(N)^{-1/(2\lambda)}\,,
\end{align}
where, with $\omega$ and $c$ as in Theorem~\ref{thm:qmc},
\begin{align*}
 C_{F,\lambda}
&\coloneqq \Bigg(\sum_{\bszero \neq \bseta \in \{0, 1\}^d} \gamma_\bseta^\lambda
\big[2\,c\, \zeta ( \omega \lambda)\big]^{|\bseta|}\Bigg)^{\frac{1}{2\lambda}}
 \Bigg(1 + \!\!\!
 \sum_{\bszero\ne\bseta \in \{0, 1\}^d}
 \frac{\big(8^{|\bseta|-1}(|\bseta|-1)!\big)^2 B_{0, \bseta}}{\gamma_\bseta}\bigg)^{\frac{1}{2}},
\\
 C_{f,\lambda}
&\coloneqq \Bigg(\sum_{\bszero \neq \bseta \in \{0, 1\}^d} \!\!\gamma_\bseta^\lambda
\big[2\,c\,\zeta ( \omega \lambda)\big]^{|\bseta|}\Bigg)^{\frac{1}{2\lambda}}
 \Bigg(\sum_{\bseta \in \{0, 1\}^d} \frac{\big(8^{|\bseta|}|\bseta|!\big)^2 B_{1,\bseta}}{\gamma_\bseta}
 \Bigg)^{\frac{1}{2}}.
\end{align*}
\end{theorem}

\begin{proof}
First for the cdf estimator, using \eqref{eq:cdf_simple} and the
definition \eqref{eq:F_N_point} of $\Fhat_N$  we can write the mean-square
error as
\[
\bbE_\bsDelta \big[|F(t) - \Fhat_N(t)|^2\big]
= \bbE_\bsDelta \bigg[\bigg|
\int_{\R^d} \Phi_0(\xi(t, \bsy)) \bsrho(\bsy) \, \rd \bsy -
Q_{d, N}(\Phi_0(\xi(t, \cdot)))
\bigg|^2\bigg].
\]
Then since $\phi$ and $\rho_0$ satisfy Assumption~\ref{asm:phi} and
Assumption~\ref{asm:phi-h} for all $\bseta \in \{0, 1\}^d$ with $q = 0$,
we can apply Theorem~\ref{thm:preint_ind} to show that the preintegrated
function $\Phi_0 (\xi(t, \cdot))$ belongs to $\calH^{\bsone}_{d}$ and its
norm is bounded by \eqref{eq:norm-cdf} with $\bsnu=\bsone$. Substituting
this norm bound into the CBC error bound \eqref{eq:cbc-err} and taking the
square root proves the desired result.

The result for the pdf estimator follows by essentially the same
argument, but with $q=1$ and using the norm bound
\eqref{eq:norm-pdf} in Theorem~\ref{thm:preint_dirac} instead of
Theorem~\ref{thm:preint_ind}.
\end{proof}

Next, we bound the RMISE on $[a, b]$. For the cdf estimator $\Fhat_{N, M}$
the mean-integrated square error (MISE) is formulated as
\[
\bbE_\bsDelta \big[ \|F - \Fhat_{N, M}\|_{L^2}^2\big]
= \bbE_\bsDelta \bigg[ \int_a^b |F(t) - \Fhat_{N, M}(t)|^2 \, \rd t \bigg]\,,
\]
and similarly for $\fhat_{N, M}$.

\begin{theorem}
\label{thm:mise} Let $d \geq 1$, $\bsnu = \bsone \in \N^d$ and $[a, b]
\subset \R$. Suppose $\phi$ and $\rho_0$ satisfy
\[
 \begin{cases}
 \mbox{Assumption~\ref{asm:phi}; and} \\
 \mbox{Assumption~\ref{asm:phi-h} for $\bseta = \bszero$ and all $q \leq d + 1$; and} \\
 \mbox{Assumption~\ref{asm:phi-h} for all $\bseta \in \{0, 1\}^d$ with $q = 0$ for the cdf case; and} \\
 \mbox{Assumption~\ref{asm:phi-h} for all $\bseta \in \{0, 1\}^d$ with $q = 1$ for the pdf case.}
 \end{cases}
\]
Suppose also that $U_t = \R^d$ for all $t \in [a, b]$. Let $Q_{d, N}$ be a
CBC-constructed randomly shifted lattice rule as in \eqref{eq:lattice} and
let $L_M$ be the Lagrange interpolation operator on $[a, b]$ based on Chebyshev points
as in \eqref{eq:L_M}.
Then for $\sigma \in \N$, specified below, and $M > \sigma$,
 the estimators $\Fhat_{N, M}$ and $\fhat_{N, M}$ in
\eqref{eq:F_NM} and \eqref{eq:f_NM} satisfy, for all $\lambda \in
(1/\omega, 1]$,
\begin{align}
\label{eq:err-F_NM}
\sqrt{ \bbE_\bsDelta \big[\|F - \Fhat_{N, M}\|_{L^2}^2\big]}
&\le C_{F,\lambda,\sigma} \Big( \phi_\mathrm{tot}(N)^{-1/(2\lambda)}
+ M^{-\sigma} \Big)
\qquad \text{for } \sigma \leq d,
\\
\label{eq:err-f_NM}
\sqrt{ \bbE_\bsDelta \big[\|f - \fhat_{N, M}\|_{L^2}^2\big]}
&\le C_{f,\lambda,\sigma} \Big( \phi_\mathrm{tot}(N)^{-1/(2\lambda)}
+ M^{-\sigma} \Big)
\qquad \text{for } \sigma \leq d - 1,
\end{align}
where
 $C_{F,\lambda,\sigma} \coloneqq \sqrt{2(b - a)}\,\max(C_{F,\lambda},C_{F,\sigma})$,
 $C_{f,\lambda,\sigma} \coloneqq \sqrt{2(b - a)}\,\max(C_{f,\lambda},C_{f,\sigma})$,
with $C_{F,\lambda}$, $C_{f,\lambda}$ as in Theorem~\ref{thm:err-point},
$\omega \in (1,2]$ as in Theorem~\ref{thm:qmc}, and
\begin{align*}
 C_{F, \sigma}&\coloneqq
 \frac{4(b - a)}{\pi}\, [3(\sigma + 1)]^\sigma\, (\sigma - 1)!\, B_{\sigma + 1, \bszero}^{1/2}, \\
 C_{f, \sigma}&\coloneqq
 \frac{4(b - a)}{\pi}\, [3(\sigma + 1)]^{\sigma + 1}\,(\sigma - 1)!\, B_{\sigma + 2, \bszero}^{1/2}.
\end{align*}
\end{theorem}

\begin{proof}
First, consider the cdf estimator $\Fhat_{N, M}$. We can split the MISE
into the QMC and interpolation components
\begin{equation}
\label{eq:mise-split}
\bbE_{\bsDelta} \big[ \|F - \Fhat_{N, M}\|_{L^2}^2\big]
\le
2\,\bbE_\bsDelta \big[ \|F - \Fhat_N\|_{L^2}^2 \big]
+ 2\,\bbE_\bsDelta \big[ \| \Fhat_N - L_M \Fhat_N \|_{L^2}^2 \big]\,.
\end{equation}

The first term in \eqref{eq:mise-split} can be bounded by the pointwise
error as follows. By Fubini's Theorem we may swap the expected value with
respect to $\bsDelta$ and the integral over $[a, b]$ to obtain
\begin{align}
\label{eq:split_qmc}
\bbE_\bsDelta \big[\|F - \Fhat_N\|_{L^2}^2 \big]
&= \int_a^b \bbE_\bsDelta \big[|F(t) - \Fhat_N(t)|^2\big] \, \rd t
 \le (b - a)\, C_{F,\lambda}^2\, \phi_\mathrm{tot}(N)^{-1/\lambda},
\end{align}
where we have substituted in the bound \eqref{eq:err-point-F}.

For the second term in \eqref{eq:mise-split} we will use the Lagrange
interpolation error bound \eqref{eq:interp_err} by first adapting the
proof of Theorem~\ref{thm:f-smooth} to show that $\Fhat_N \in W^{\sigma +
1, \infty}[a, b]$ for $\sigma \leq d$ and all random shifts $\bsDelta$. Differentiating
\eqref{eq:F_N_point} with respect to $t$ then substituting in the formula
\eqref{eq:hyp3} for $q=\sigma+1$ gives
\begin{equation} \label{eq:d-Fhat}
  \Fhat_N^{(\sigma+1)}(t)
= \frac{1}{N}\sum_{n=0}^{N - 1} \pd{\sigma+1}{}{t^{\sigma+1}} \Phi_0\big(\xi(t,\bstau_n^\bsDelta)\big)
= \frac{1}{N}\sum_{n=0}^{N - 1} \sum_{j= 1}^{J_{\sigma+1,\bszero}}
h_{\sigma+1, \bszero}^{[j]} (t,\bstau_n^\bsDelta) \,,
\end{equation}
with $J_{\sigma+1,\bszero} \le 3^\sigma\,\sigma!\,$, where we emphasized
the explicit dependence of each transformed QMC point on the random shift with the
superscript $\bsDelta$. Hence, we can apply \eqref{eq:interp_err} to
obtain
\begin{align}
\label{eq:split_interp1}
\bbE_\bsDelta \big[\|\Fhat_N - L_M \Fhat_N\|_{L^2}^2\big]
 &\le (b - a)\, \bbE_\bsDelta \big[\|\Fhat_N - L_M \Fhat_N\|_{L^\infty}^2 \big]
\nonumber\\
&\le (b - a)\,\bigg(\frac{4}{\pi\sigma (M - \sigma)^{\sigma}} \bigg)^2\,
\bbE_\bsDelta \big[\|\Fhat_N^{(\sigma + 1)}\|_{L^1}^2\big] \nonumber\\
&\le (b - a)\,\bigg(\frac{4\,(\sigma + 1)^\sigma}{\pi\sigma}\bigg)^2\,
\bbE_\bsDelta \big[\|\Fhat_N^{(\sigma + 1)}\|_{L^1}^2\big] \, M^{-2\sigma} \,,
\end{align}
where we used the easily verified inequality $M-\sigma\ge M/(\sigma+1)$
for all $M\ge\sigma+1$.

To bound the expected value in \eqref{eq:split_interp1}, we use the
formula \eqref{eq:d-Fhat}, the Cauchy--Schwarz inequality for integral and
sum, and Fubini's Theorem to obtain
\begin{align} \label{eq:expect}
 \bbE_{\bsDelta} \big[ \|\Fhat_N^{(\sigma+1)}\|_{L^1}^2\big]
 &\le \frac{(b-a)\,J_{\sigma+1,\bszero}}{N} \sum_{n=0}^{N-1} \sum_{j=1}^{J_{\sigma+1,\bszero}}
 \int_a^b \int_{[0,1]^d} | h_{\sigma+1, \bszero}^{[j]}(t,\bstau_n^\bsDelta)|^2 \,\rd\bsDelta \,\rd t \nonumber\\
 &= \frac{(b-a)\,J_{\sigma+1,\bszero}}{N} \sum_{n=0}^{N-1} \sum_{j=1}^{J_{\sigma+1,\bszero}}
 \int_a^b \int_{\bbR^d} | h_{\sigma+1, \bszero}^{[j]}(t,\bsy)|^2\,\bsrho(\bsy) \,\rd\bsy \,\rd t \nonumber\\
 &\le (b-a)^2\, J_{\sigma+1,\bszero}^2\, B_{\sigma+1, \bszero}
 \le [(b-a)\,3^\sigma\,\sigma!]^2\, B_{\sigma+1, \bszero},
\end{align}
where we made a change of variables and then used the upper bound
\eqref{eq:h-int}, which is assumed to hold for all $q = \sigma + 1 \leq |\bsnu| + 1 = d + 1$,
as well as the bound $J_{\sigma+1,\bszero}\le
3^\sigma\,\sigma!\,$. Substituting \eqref{eq:expect} into
\eqref{eq:split_interp1}, we conclude that
\begin{equation}
\label{eq:split_interp}
\bbE_\bsDelta \big[\|\Fhat_N - L_M \Fhat_N\|_{L^2}^2\big]
 \le  (b - a)\, C_{F,\sigma}^2\, M^{-2\sigma}\,,
\end{equation}
with $C_{F,\sigma}$ as defined in the theorem. Substituting
\eqref{eq:split_qmc} and \eqref{eq:split_interp} into
\eqref{eq:mise-split}, we obtain the RMISE bound \eqref{eq:err-F_NM} for
the cdf estimator.

The result for the pdf estimator follows by essentially the same argument.
The key difference is that we must instead bound the norm
$\|\fhat_N^{(\sigma+1)}\|_{L^1}$ for $\sigma + 1 \leq |\bsnu| = d$. 
Similar to the relationship $f = F'$, it
follows from \eqref{eq:dt_Phi0} that $\fhat_N$ is the derivative with
respect to $t$ of $\Fhat_N$. Thus
\begin{align*}
 \fhat_N^{(\sigma+1)}(t) = \Fhat_N^{(\sigma+2)}(t)
= \frac{1}{N}\sum_{n=0}^{N - 1} \pd{\sigma+2}{}{t^{\sigma+2}}
\Phi_0(\xi(t, \bstau_n^\bsDelta))
= \frac{1}{N}\sum_{n=0}^{N - 1} \sum_{j = 1}^{J_{\sigma+2,\bszero}}
h_{\sigma+2, \bszero}^{[j]} (t, \bstau_n^\bsDelta) \,,
\end{align*}
with $J_{\sigma+2,\bszero} \le 3^{\sigma+1}\,(\sigma+1)!\,$. Following the
same steps as before, we obtain
\[
 \bbE_\bsDelta\big[
 \|\fhat_N^{(\sigma+1)}\|_{L^1}^2\big] \le [(b - a)\,3^{\sigma+1}\,(\sigma+1)!]^2\, B_{\sigma + 2, \bszero}\,,
\]
and eventually arrive at the RMISE bound \eqref{eq:err-f_NM} for the
pdf estimator.
\end{proof}

\begin{remark}
In Theorem~\ref{thm:mise}, we assume the minimal smoothness required
to apply the QMC theory from Theorem~\ref{thm:qmc} (see also \cite{NK14}), i.e., $\bsnu = \bsone$,
which in most cases is sufficient to also handle the interpolation error.
For the special case $d = 1$, the interpolation component of the error converges linearly for $\Fhat_{N, M}$,
i.e., $\sigma = 1$ in \eqref{eq:err-F_NM}, but does not converge for $\fhat_{N, M}$, i.e., $\sigma = 0$ in 
\eqref{eq:err-f_NM}. Under stricter smoothness assumptions, one can of course prove higher convergence
rates for the interpolation error when $d = 1$.
However, in this case, after preintegration one only needs to approximate a one-dimensional
integral, for which it is more appropriate to use a one-dimensional quadrature rule, e.g., Gauss--Hermite,
instead of QMC. Such a rule would additionally exploit any higher smoothness assumptions
to obtain higher-order convergence for the quadrature error.
\end{remark}

Theorem~\ref{thm:mise} implies that we can take $\sigma$ up to $d - 1$ 
(or $d$ for the cdf approximation)
and obtain a very fast convergence rate in terms of $M$. However, as $\sigma$
increases the constant increases significantly. Hence, in practice one
should take a moderate value for $\sigma$, e.g., around 2--5.

To see how Theorem~\ref{thm:mise} applies in practice, let now $N$ be a prime or a prime power,
in which case $\phi_\mathrm{tot}(N) \sim N$.
Then to balance the QMC and interpolation error,
Theorem~\ref{thm:mise} implies that we should take $M \sim N^{1/\sigma}$.
The final result is that the estimators converge at a rate arbitrarily close to $1/N$.
It is summarised in the following Corollary.

\begin{corollary}
\label{cor:error} Let $d \geq 2$ and suppose that the conditions in Theorem~\ref{thm:mise}
hold. Let $N$ be a prime power and choose $M \sim N^{1/\sigma}$ for a
moderate $\sigma \leq d - 1$. Then, for $\epsilon > 0$ there exist
constants $C_{F, \epsilon}$, $C_{f, \epsilon} < \infty$ such that the
error of the cdf and pdf estimators satisfy
\begin{align*}
\sqrt{\bbE_{\bsDelta} \big[ \|F - \Fhat_{N, M}\|_{L^2}^2\big]}
&\,\leq\, C_{F, \epsilon} \, N^{-1 + \epsilon},
\quad
\\
\sqrt{\bbE_{\bsDelta} \big[ \|f - \fhat_{N, M}\|_{L^2}^2\big]}
&\,\leq\, C_{f, \epsilon} \, N^{-1 + \epsilon}\,.
\end{align*}
\end{corollary}

\section{Numerical results}

\label{sec:numerics}
To test the method, we consider approximating the cdf and pdf of a random
variable $X \in \R$ given by a sum of $d + 1$ log-normals,
\begin{equation}
\label{eq:X_lognorm}
 X \,=\, \sum_{i = 0}^d \exp ( W_i)
 \,=\, \sum_{i = 0}^d \exp( \bsA_i \bsY)
 \,\eqqcolon\, \phi(\bsY),
\end{equation}
where $\bsW = (W_i)_{i = 0}^d$ is a $(d+1)$-dimensional multivariate
normal vector with mean~$\bszero$ and covariance $\Sigma$.
In the second equality, we have factorised the
covariance matrix as $\Sigma = AA^\top$ and made the change of variables
$\bsY = A^{-1} \bsW$, so that $\phi$ is a function of
the $(d + 1)$-dimensional standard normal vector $\bsY = (Y_i)_{i = 0}^d$.
In \eqref{eq:X_lognorm}, $\bsA_i$ denotes the $i$th row of the matrix factor $A$.
We use the \emph{principal components} or \emph{PCA} factorisation,
which is based on the eigendecomposition of $\Sigma$ with
the eigenvalues ordered in nonincreasing value.
Clearly, $X$ fits the setting of this paper with
$\rho_i(y_i) = e^{-y_i^2/2}/\sqrt{2\pi}$.

We test the method for two covariance matrices. The first example is
in $d + 1 = 32$ dimensions with covariance matrix $\Sigma^{(1)}$ and
the second example takes $d + 1 = 64$ and a covariance matrix $\Sigma^{(2)}$ with
entries that are decaying in value:
\begin{align*}
\Sigma^{(1)}_{i, j} =
\begin{cases}
1 & \text{for } i = j\,,\\
\frac{1}{2} & \text{for } i \neq j\,,
\end{cases}
\quad \text{and} \quad
\Sigma^{(2)}_{i, j} = \frac{1}{\max(i, j)}\,.
\end{align*}
It can easily be verified that Assumption~\ref{asm:phi} and
Assumption~\ref{asm:phi-h}, with the weight functions
$\psi_i(y) = e^{-\delta y^2/ 2}$ for $\delta \in (0, \tfrac{1}{2})$,
are satisfied for both covariance matrices. For this choice of
$\psi_i$, Table~3 in \cite{KSWW10} indicates that
\eqref{eq:omega} holds with $\omega = 2(1 - \delta)$, giving a QMC
convergence rate of $1 - \delta$. Choosing the weight parameters
$\{\gamma_\bseta\}$ and performing a full error analysis that is explicit
in the dependence on the dimension is left for future work.

For the QMC approximations we use embedded lattice rules given by
the CBC construction from \cite{CKN06} and which are available at
\cite{Kuo_lattice}. Although there is no accompanying theory, these rules
have been shown to work well in practice. For $\Sigma^{(1)}$, we use the
generating vectors \texttt{lattice-38005-1024-1048576.5000}, constructed
using equal product weights $\gamma_i = 0.05$,
 and for $\Sigma^{(2)}$, we use
\texttt{lattice-39102-1024-1048576.3600}, constructed using decaying
product weights $\gamma_i = 1/i^2$.

The final estimate is the average over $R = 32$ random shifts, and we
estimate the RMSE of this average by the sample standard error over
the random shifts. For a fair comparison, each MC approximation
uses $R\times N$ points in total and the RMSE is estimated by the sample
standard error over all MC realisations. For each $\bstau_n$,
the value of $\xi(t,\bstau_n)$ is computed using Newton's method with
a tolerance of $10^{-10}$. All computations were run on the computational
cluster Katana \cite{katana}
at UNSW Sydney.

Convergence results for the cdf and pdf at the point $t = 60$ for both
covariance matrices are given in Figures~\ref{fig:pointwise1} and \ref{fig:pointwise2} for $N =
2^{10}, 2^{11}, \ldots, 2^{20}$, where we plot the approximate relative RMSE
(the RMSE divided by the estimated value). We see clearly that preintegration
drastically improves the empirical results for QMC, especially for the
matrix $\Sigma^{(2)}$ which has decaying eigenvalues. The results for
$\Sigma^{(1)}$ are similar to those presented in \cite{BotSalMac19}.
A possible explanation for the better results observed
for $\Sigma^{(2)}$, compared to $\Sigma^{(1)}$, is that although formally the dimension is larger,
this problem may have a lower effective dimension.
Since the eigenvalues of $\Sigma^{(2)}$ are decaying, this suggests that
the variables are also decaying in importance, whereas
the largest eigenvalue of $\Sigma^{(1)}$ is simple and
the remaining $d = 31$ eigenvalues are equal, suggesting that
after preintegration the remaining variables are all equally important.

\begin{figure}[t]
\centering
\includegraphics[scale=0.4]{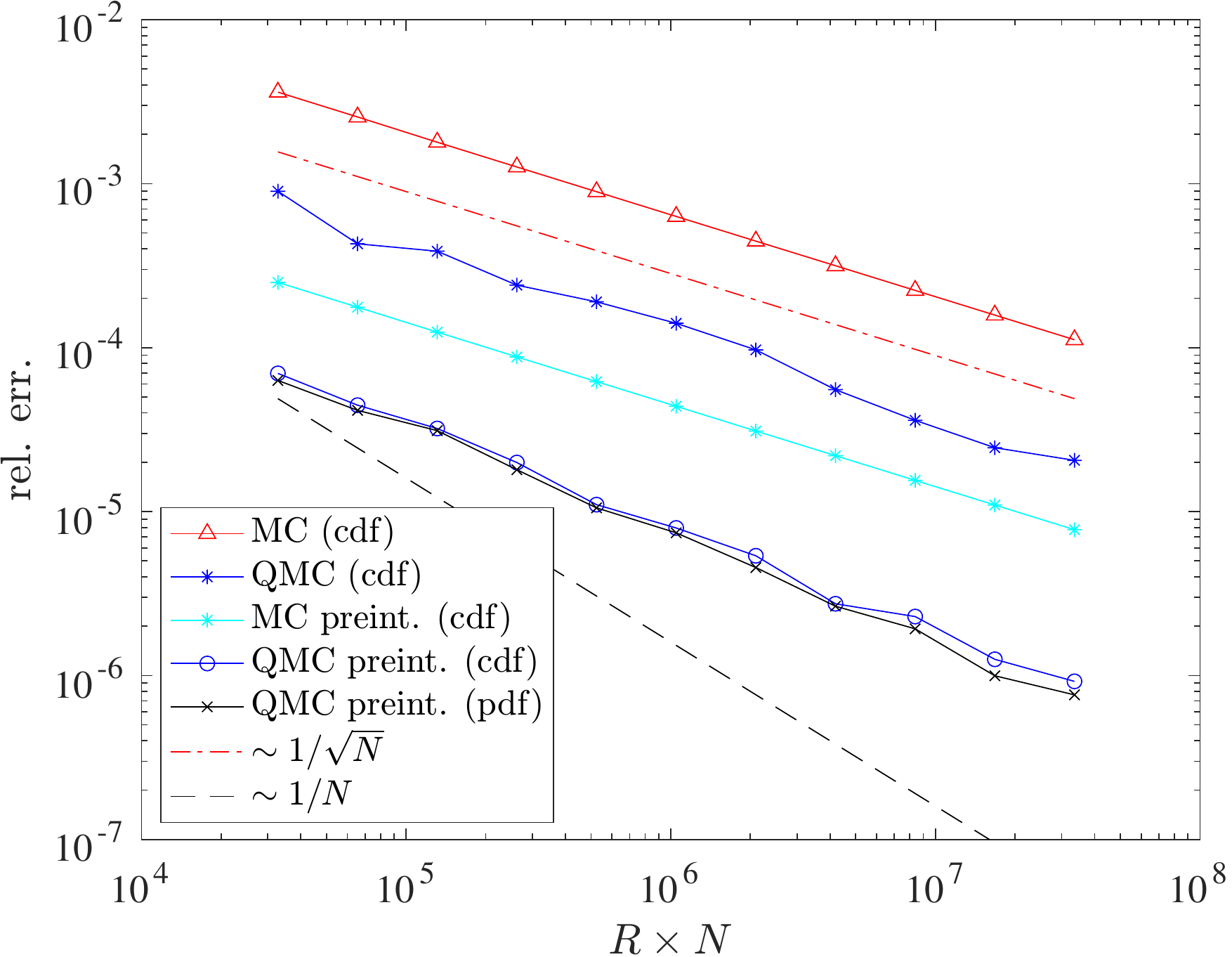}
\caption{Relative RMSE convergence in $N$ for MC and QMC, with and without preintegration,
for $F(60) = \bbP[ X \leq 60]$, and also QMC with preintegration
for $f(60)$, for $\Sigma^{(1)}$.}
\label{fig:pointwise1}
\end{figure}
\begin{figure}[t]
\centering
\includegraphics[scale=0.4]{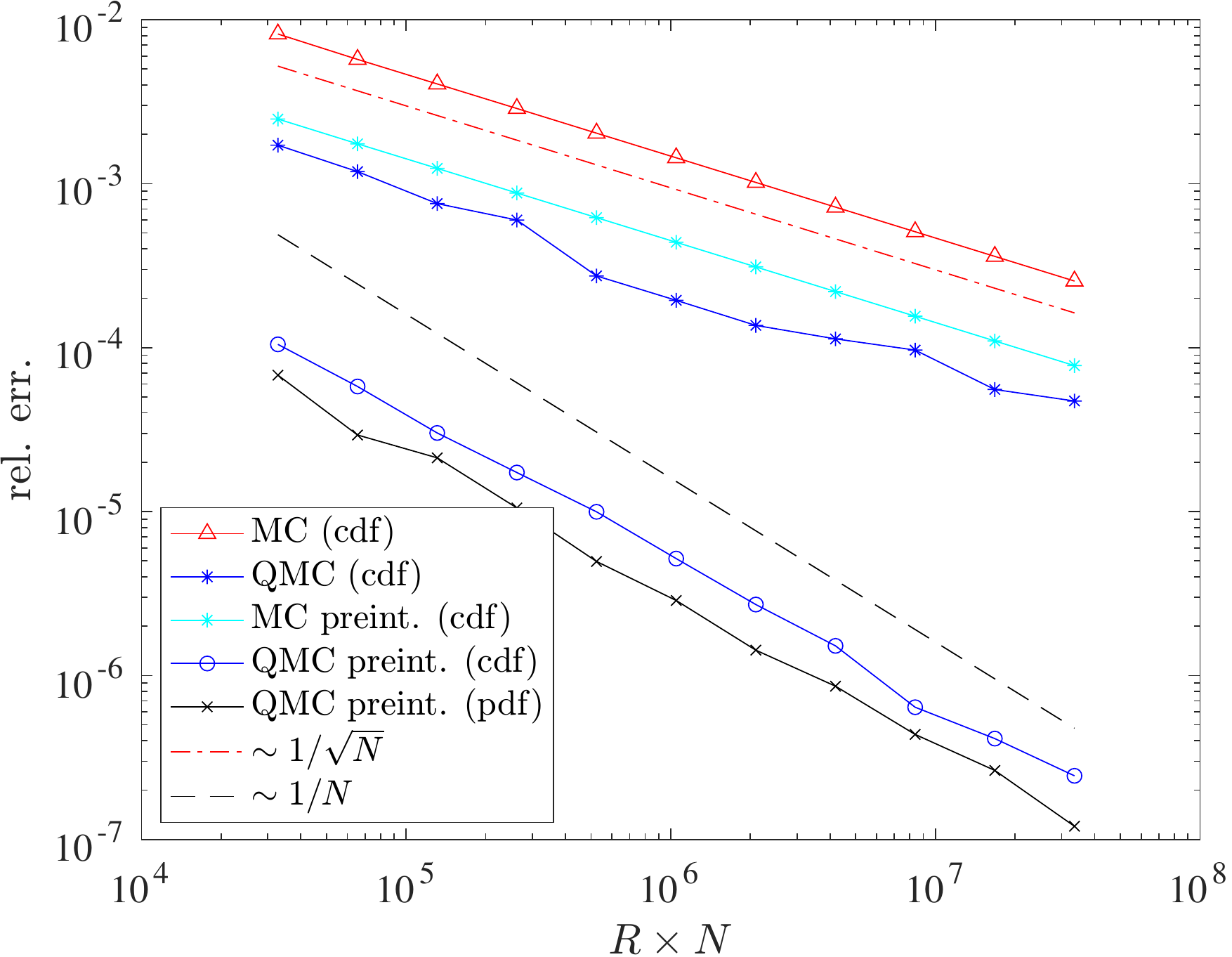}
\caption{Relative RMSE convergence in $N$ for MC and QMC, with and without preintegration,
for $F(60) = \bbP[ X \leq 60]$, and also QMC with preintegration
for $f(60)$, for $\Sigma^{(2)}$.}
\label{fig:pointwise2}
\end{figure}


Tables~\ref{tab:time1} and \ref{tab:time2} give the CPU times
for the QMC and QMC with preintegration approximations for the cdf at $t = 60$,
as well as the QMC with preintegration approximation of the pdf at $t = 60$,
for $\Sigma^{(1)}$ and $\Sigma^{(2)}$, respectively.
The timing tests were run on a single processor
for $N = 2^{13}, 2^{14}, \ldots, 2^{20}$ with a single random shift.
As expected, the CPU time increases linearly with $N$.
In the last rows, we give the increase factor of the CPU time
for QMC with preintegration compared to plain QMC for the cdf.
Since this factor is around 2 in all cases, it is much less compared to the
the error reduction observed in Figures~\ref{fig:pointwise1} and \ref{fig:pointwise2},
which ranges from 10 to over 100.
This demonstrates that preintegration is well worth the slight increase in cost.

\begin{table} [t]
\centering
\begin{tabular}{|c||r|r|r|r|r|r|r|r|}
\hline
$N$  & 
 $2^{13}$ & $2^{14}$ & $2^{15}$ & $2^{16}$ & $2^{17}$ & $2^{18}$
& $2^{19}$ & $2^{20}$
\\
\hline
QMC (cdf) &       
    0.05 &    0.11 &    0.23 &    0.46 &    0.95 &    1.92 &    3.59 &    7.17 \\
QMC preint.~(cdf) &  
    0.13 &    0.25 &    0.50 &    1.01 &    2.03 &    4.11 &    8.09 &   16.12
\\
QMC preint.~(pdf) &    
    0.09 &    0.18 &    0.34 &    0.72 &    1.36 &    2.89 &    5.41 &   10.88
    \\
\hline
increase factor (cdf) &     
 2.6 &    2.3 &    2.2 &    2.2 &    2.1 &    2.1 &    2.3 &    2.2
 \\
 \hline
\end{tabular}
\label{tab:time1}
\caption{CPU times for QMC approximations with a single random shift for $\Sigma^{(1)}$.}
\begin{tabular}{|c||r|r|r|r|r|r|r|r|}
\hline
$N$  & 
$2^{13}$ & $2^{14}$ & $2^{15}$ & $2^{16}$ & $2^{17}$ & $2^{18}$ & $2^{19}$ & $2^{20}$
\\
\hline
QMC (cdf) &       
     0.08 &    0.17 &    0.32 &    0.67 &    1.28 &    2.60 &    5.14 &   10.22
\\
QMC preint.~(cdf) &  
    0.15 &    0.31 &    0.63 &    1.31 &    2.49 &    5.03 &   10.04 &   19.83
\\
QMC preint.~(pdf) &    
    0.12 &    0.23 &    0.49 &    0.92 &    1.92 &    3.71 &    7.49 &   14.68
    \\
\hline
increase factor (cdf) &     
    1.9 &    1.8 & 2.0 & 2.0 &    1.9 &    1.9 &    2.0 &    1.9
\\
\hline
\end{tabular}
\caption{CPU times for QMC approximations with a single random shift for $\Sigma^{(2)}$.}
\label{tab:time2}
\end{table}

In Figure~\ref{fig:cdfpdf}, we plot the QMC with preintegration estimators
for both the cdf $F$ (left) and pdf $f$ (right)
on the interval $[40,100]$. Figure~\ref{fig:rmise} then plots the
convergence of the RMISE. The
approximations of the cdf and pdf use $N = 2^{20}$ QMC points averaged
over $R=32$ random shifts and degree $M = 42$ Lagrange interpolation based
on Chebyshev points. For the RMISE we use $N = 2^{10}, 2^{11}, \ldots,
2^{19}$ with interpolation degree $M = \lceil N^{1/4} \rceil + 10$ and $R =
32$ random shifts. In this way, the number of interpolation points is coupled
to the number of QMC points as in Corollary~\ref{cor:error},
with the choice $\sigma = 4$.
To estimate the RMISE of $\Fhat_{M, N}$ and $\fhat_{M, N}$,
we first approximate the $L^2$
error by comparing each approximation to an approximation with much higher
precision (i.e., $M=42$, $N=2^{20}$, $R=32$) treated as the true cdf or
pdf. Then we approximate the mean with respect to $\bsDelta$ by averaging the $L^2$
error over the random shifts.
This captures both the QMC and interpolation contributions to the RMISE
As expected from Corollary~\ref{cor:error},
we observe almost $1/N$ convergence for the RMISE. This demonstrates that the method is an
effective practical strategy.

\begin{figure}[t]
\centering
\includegraphics[scale=0.3]{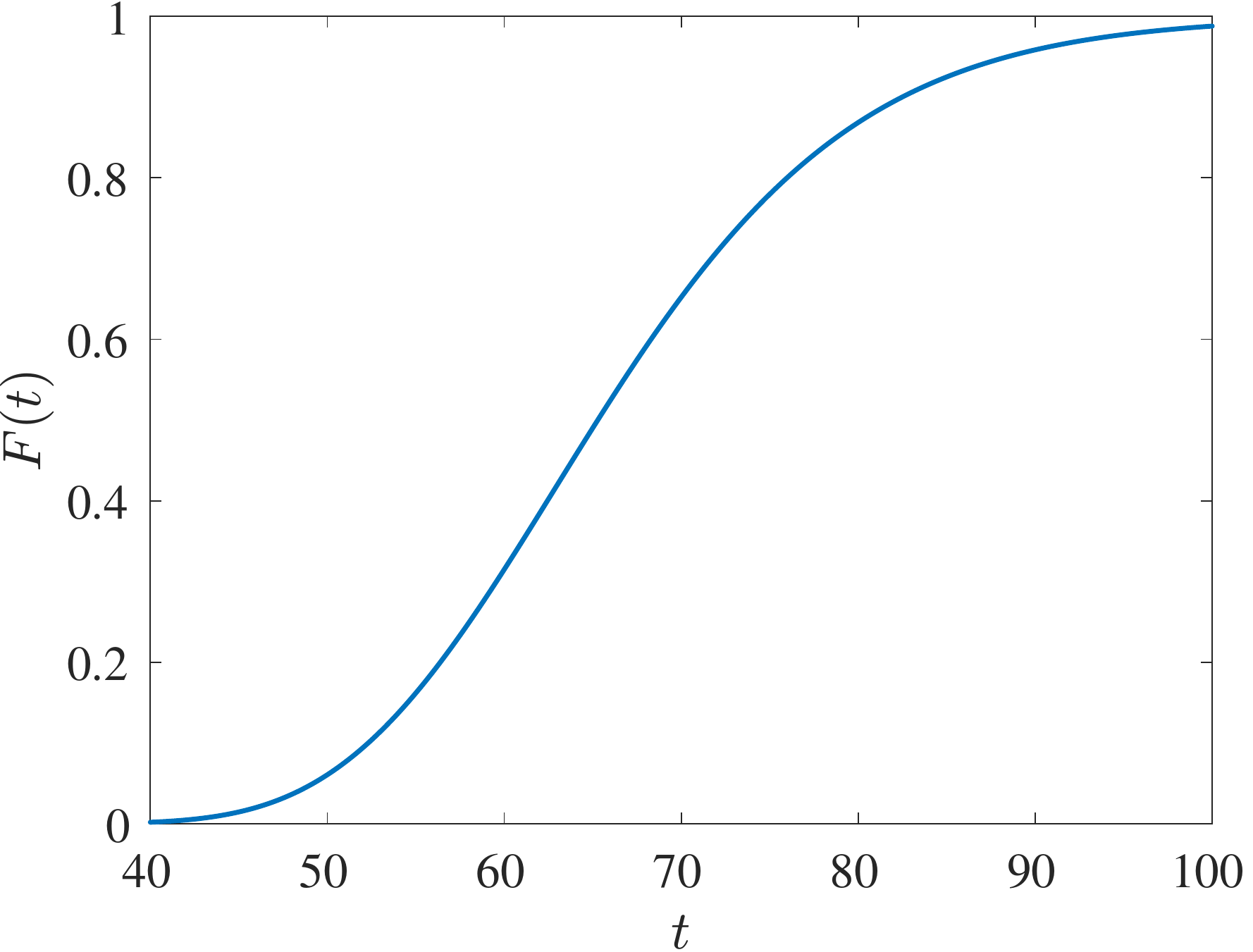}
~~
\includegraphics[scale=0.3]{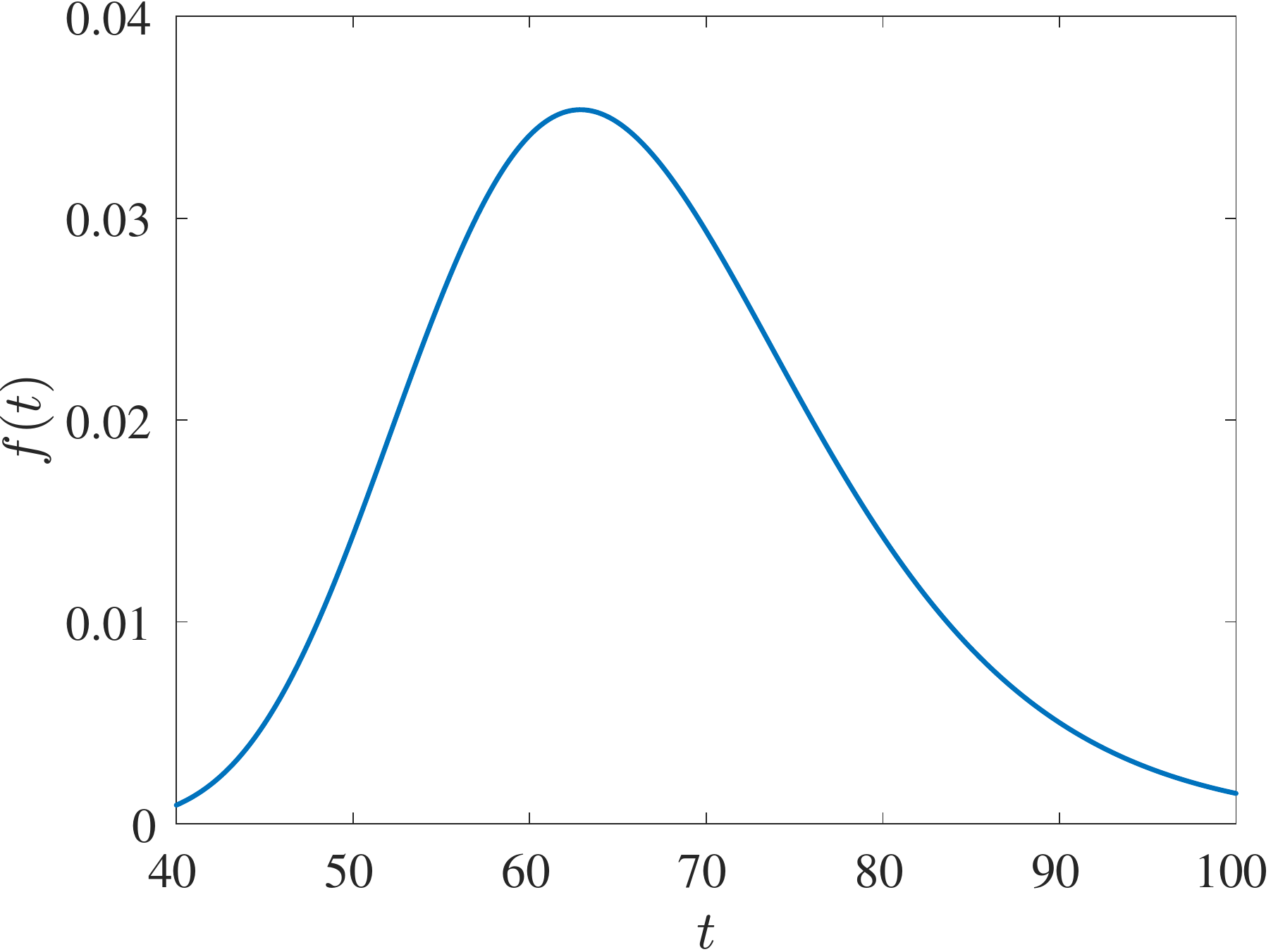}
\caption{Approximate cdf (left) and approximate pdf (right) for $\Sigma^{(2)}$.}
\label{fig:cdfpdf}
\end{figure}
\begin{figure}[!h]
\centering
\includegraphics[scale=0.4]{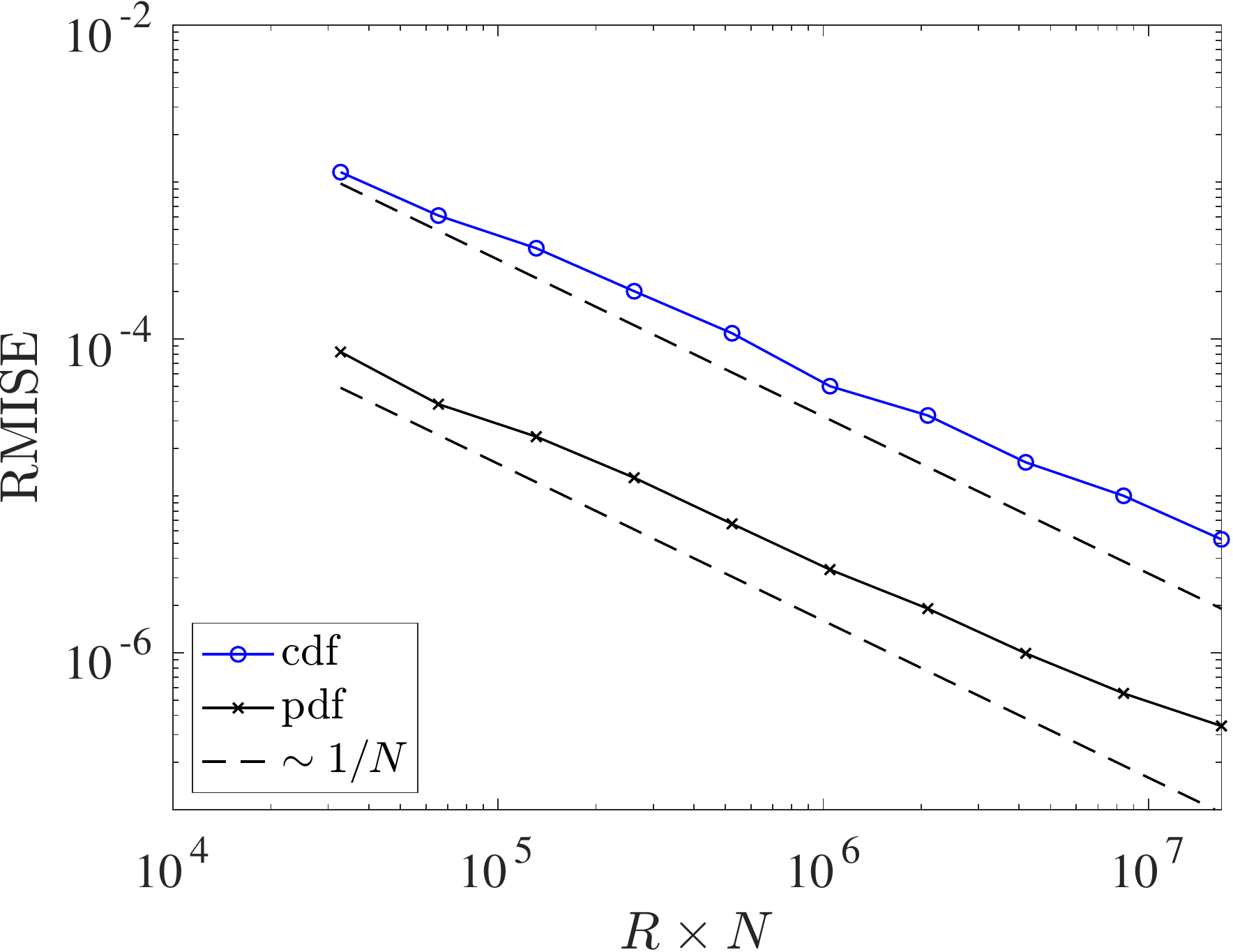}
\caption{Convergence of RMISE in $N$ for $\Sigma^{(2)}$.}
\label{fig:rmise}
\end{figure}


\bibliographystyle{plain}

\begin{appendix}

\section{Illustrative examples}

\subsection{Some functions in $\calH^{\bsone}_{d}$ are not in BVHK after transformation}
\label{app:1}

Mapping a weighted integral of an unbounded function $g$ on $\R^{d}$
to the $d$-dimensional unit cube as in \eqref{eq:qmc} clearly results in
an unbounded integrand over $[0,1]^{d}$, since the closure of the range of
$g$ is unchanged by the transformation. This precludes membership of BVHK
for the transformed integrand. For example, letting $d = 1$ and
$g(y)=y^2$, and taking $\rho$ to be the standard normal Gaussian density,
results in an unweighted integral over $[0,1]$ of an unbounded
function, namely $\Phi^{-1}(u)^2$. An unbounded function of a single
variable does not have bounded variation. The argument applies equally to
any non-constant polynomial on $\R^d$, and to any unbounded function $g$
that is integrable after multiplication by $\rho$.

On the other hand, choosing the weight function $\psi(y) = e^{-|y|}$,
it is easily seen that $g(y) = y^2$, or any polynomial, belongs to
$\calH^{\bsone}_1$. Furthermore, for this pair of $\rho$ and~$\psi$,
Theorem~\ref{thm:qmc} holds with $\omega$ arbitrarily close to 2, and
hence we get nearly first order convergence.

Thus the setting used in the present paper allows a wider class of
problems than assuming that the transformed function is in BVHK.

\subsection{Smoothness fails for preintegrated functions if the random variables $Y_i$ have compact support}
\label{app:2}

The following example, defined on $\R^4$, conforms with the definitions in the Introduction for $d=3$, except that in this case the underlying independent random variables $Y_i$ have support $[0,1]$, and have uniform distribution, corresponding to $\rho_i(y) = 1$ for $y \in [0, 1]$ and
$\rho_i(y) = 0$ elsewhere.   We take
\begin{equation*}
\phi(\bsy) \coloneqq \phi(y_0, y_1, y_2, y_3) =y_0 - \frac{y_1}{3}- \frac{y_2}{3}- \frac{y_3}{3}+\frac{1}{2}, \quad \bsy  \in [0,1]^4.
\end{equation*}
Preintegration with respect to $y_0$ of the indicator function composed with  $\phi$  yields
\begin{align}\label{eq:preintexample}
 \int_0^1 \ind(\phi(\bsy))\,\rd y_0
 &= \int_0^1\ind \Big(y_0 - \frac{y_1}{3}- \frac{y_2}{3}- \frac{y_3}{3}+\frac{1}{2}\Big)\,\rd y_0
 \nonumber\\
 & = \int_{(\frac{y_1}{3}+ \frac{y_2}{3}+ \frac{y_2}{3}-\frac{1}{2})_+}^1  1 \,\rd y_0
 = 1 - \Big(\frac{y_1}{3}+ \frac{y_2}{3}+ \frac{y_3}{3}-\frac{1}{2}\Big)_+\,,
\end{align}
which is continuous on $[0,1]^3$, but has a discontinuous gradient
across the plane
\begin{equation}\label{eq:plane}
y_1 + y_2 + y_3 = \frac{3}{2}.
\end{equation}

The preintegrated function \eqref{eq:preintexample} is not in the class
BVHK, and hence does not have the smoothness assumed for QMC.  To see
this, for $n$ a positive even integer, consider all the  cubes of edge
length $1/n$ with one vertex at $(\tfrac{i_1}{n}, \tfrac{i_2}{n},
\tfrac{i_3}{n})$ and the diagonally opposite vertex at $(\tfrac{i_1+1}{n},
\tfrac{i_2+1}{n}, \tfrac{i_3+1}{n})$, with $i_1, i_2, i_3 \in \{0, 1,
\ldots, n-1\}$ and such that $i_1 + i_2 + i_3 = \tfrac{3}{2}n - 2$.  For
each such cube, the vertex  $(\tfrac{i_1 + 1}{n}, \tfrac{i_2 + 1}{n},
\tfrac{i_3 + 1}{n})$ satisfies
\[
\frac{y_1}{3}+ \frac{y_2}{3}+ \frac{y_3}{3} - \frac{1}{2} = \frac{1}{3n} > 0,
\]
and hence $(\frac{y_1}{3}+ \frac{y_2}{3}+ \frac{y_3}{3}-\frac{1}{2})_+$
has the value $\tfrac{1}{3n}$ at that vertex, while having the
value~$0$ at the other seven vertices, since they lie on or
below the plane \eqref{eq:plane}. Thus each small cube contributes
$\tfrac{1}{3n}$ to the variation in the sense of Vitali \cite{Nie92}.
Moreover, because the plane \eqref{eq:plane} is a $2$-dimensional manifold
with a non-trivial intersection with the unit cube, it is easy to see that
there are of exact order $n^2$ such cubes in the unit cube. Thus the total
variation in the sense of Vitali has a lower bound of exact order $n$.
Letting $n\to\infty$, it follows that the variation in the sense of
Vitali, and hence also of Hardy and Krause, is infinite. Similar
results also hold in higher dimensions.

The paper \cite{GKS10} shows that for the case of the cube, as in this
example, preintegration needs to be performed repeatedly if the mixed
derivative smoothness is to be improved, with each successive
preintegration yielding at most one additional order of smoothness. This
implies that, in general, for the unit cube one must perform
preintegration with respect to at least $d/2$ different variables to
obtain first-order mixed derivative smoothness.

\subsection{The monotonicity condition is necessary}
\label{app:3}

Consider the bivariate  function $g(y_1,
y_2)\coloneqq\sqrt[m]{(y_1-y_2)_+}$, for $m\ge 2$, on $[0,1]^2$.  This is
a simplified model (simplified in being restricted to the unit square
instead of $\R^d$) of the typical singularity shown in \cite{GKS22b} to
arise when monotonicity with respect to the preintegration variable fails.
A lower bound on the variation in the sense of Vitali, and hence of
Hardy and Krause, can be obtained by taking a uniform  partition of edge
length $1/n$ of the unit square, and then considering the contribution to
the variation in the sense of Vitali from only those squares of size $1/n$
that are bisected by the main diagonal. From the values of $g$ at the four
vertices, each such square contributes $|0 - 0 + 0  - \sqrt[m]{1/n}| =
\sqrt[m]{1/n}$. Since there are $n$ such squares, a lower bound on the
variation of $g$ is $n^{1-1/m}$, which for all $m\ge 2$ is unbounded as $n
\to \infty$. Thus $g$ is not in the class BVHK.

\subsection{Implications for analysis of QMC after preintegration}

The results in \cite{GKS10} and the example in Section~\ref{app:2}
show that preintegration is not guaranteed to be an effective method on compact domains, 
e.g., the unit cube.
Hence, in this work we assume that the random variables $Y_i$, $i = 0, 1, \ldots d$, have support on the whole real line,
resulting in an integration problem defined on $\R^{d + 1}$, and we then use the theory from \cite{GKLS18}, which proves
that on $\R^{d + 1}$ one step of preintegration is sufficient (under certain conditions).
Furthermore, it is well-known that for problems on unbounded domains, 
the strategy of mapping back to the unit cube and using the Koksma--Hlawka inequality has the
drawback that it cannot handle unbounded integrands, because they are not in BVHK
(see also Section~\ref{app:1}). Hence, in this paper we use the setting introduced in \cite{NK14} (specifically an equivalent space),
which can handle unbounded integrands by an appropriate choice
of the weight functions $\{\psi_i\}$.
The BVHK setting may allow for certain functions
that do not have square-integrable mixed derivatives, as assumed in our setting.

In summary, preintegration is most effective for problems on unbounded domains,
for which the results in \cite{NK14} provide the appropriate setting to perform QMC error analysis.
On the other hand, Section~\ref{app:3} and \cite{GKS22b} show that if the monotonicity 
condition fails then there may exist a square-root singularity, 
which can neither be handled by BVHK nor by our setting.

\section{Technical results}
\label{app:tech}

\begin{lemma} \label{lem:count-y}
Let $d\ge 1$, $\bsnu\in \bbN_0^d$, and $[a,b]\subset\bbR$. Suppose that
$\phi$ and $\rho_0$ satisfy Assumption~\ref{asm:phi}, and recall the
definitions of $U_t$, $\xi$ and $V$ in \eqref{eq:Ut}, \eqref{eq:xi} and
\eqref{eq:V}, respectively. For any $q\in\bbN_0$ and $\bseta\le\bsnu$
satisfying $|\bseta|+q \le |\bsnu|+1$, we consider functions
$h_{q,\bseta}: V \to \R$ of the form \eqref{eq:h-form}. Then for any $i\in
\{1:d\}$, we can write
\begin{align*}
 D^i h_{q,\bseta}(t,\bsy)
 = \sum_{k=1}^{K_{q,\bseta}} h_{q,\bseta+\bse_i}^{[k]}(t,\bsy)
 \quad\mbox{with}\quad
  K_{q,\bseta} \le 8|\bseta| +6q-3,
\end{align*}
where each function $h_{q,\bseta+\bse_i}^{[k]}$ is of the form
\eqref{eq:h-form} with $\bseta$ replaced by $\bseta+\bse_i$.
\end{lemma}

\begin{proof}
For any $i\in \{1:d\}$ and $h_{q,\bseta}(t,\bsy) = h_{q,\bseta,
(r,\bsalpha,\beta)}(t,\bsy)$ from \eqref{eq:h-form} we have
\begin{align*}
 &D^i h_{q,\bseta, (r,\bsalpha,\beta)}(t,\bsy)
 = D^i \bigg(
     \underbrace{\frac{(-1)^r}{[D^0\phi(\xi(t,\bsy),\bsy)]^{r + q}}\vphantom{\prod_{\ell=1}^2}}_{=:\, T_1(t,\bsy)}\;
     \underbrace{\rho_0^{(\beta)}(\xi(t,\bsy)) \vphantom{\prod_{\ell=1}^2}}_{=:\, T_2(t,\bsy)}\;
     \underbrace{\prod_{\ell=1}^{r} D^{\bsalpha_\ell} \phi(\xi(t, \bsy), \bsy)}_{=:\, T_3(t,\bsy)}
     \bigg).
\end{align*}
Using the chain rule and substituting $D^i\xi(t,\bsy) = -
D^i\phi(\xi(t,\bsy),\bsy)/D^0\phi(\xi(t,\bsy),\bsy)$ (see
\eqref{eq:Di-xi}), and then simplifying our notation by suppressing the
dependence on $t$ and $\bsy$, we obtain
\begin{align*}
 D^i T_1(t,\bsy)
 &= \frac{-(r+q)\,(-1)^r\big[D^iD^0\phi(\xi(t,\bsy),\bsy) + D^0D^0\phi(\xi(t,\bsy),\bsy)\, D^i\xi(t,\bsy)\big]}
      {[D^0\phi(\xi(t,\bsy),\bsy)]^{r+q+1}} \\
 &= \frac{(r+q)\,(-1)^{r+1}\,D^{\bse_0 +\bse_i}\phi(\xi)}{[D^0\phi(\xi)]^{r+1+q}}
    + \frac{(r+q)\,(-1)^{r+2}\,D^{2\bse_0}\phi(\xi)\,D^{\bse_i}\phi(\xi)}{[D^0\phi(\xi)]^{r+2+q}},
 \\
 D^i T_2(t,\bsy)
 &= \rho_0^{(\beta+1)}(\xi(t,\bsy))\,D^i\xi(t,\bsy)
 = -\frac{\rho_0^{(\beta+1)}(\xi)\,D^{\bse_i}\phi(\xi)}{D^0\phi(\xi)}, 
\\
D^i T_3(t,\bsy)
 &= \!\sum_{m=1}^{r}\!\! \!\big[
 D^i\!D^{\bsalpha_m}\!\phi(\xi(t,\bsy),\bsy) \!+\! D^0\! D^{\bsalpha_m}\!\phi(\xi(t,\bsy),\bsy) D^i\!\xi(t,\bsy)\big]
\!\!\!\prod_{\substack{\ell= 1\\\ell\ne m}}^{r}\!\!\! D^{\bsalpha_\ell}\!\phi(\xi(t,\bsy),\bsy) \\
 &= \sum_{m=1}^{r}\! \bigg[
 D^{\bsalpha_m+\bse_i}\phi(\xi) - \frac{D^{\bsalpha_m+\bse_0}\phi(\xi)\, D^{\bse_i}\phi(\xi)}{D^0\phi(\xi)}
 \bigg]
 \prod_{\substack{\ell= 1\\\ell\ne m}}^{r}\! D^{\bsalpha_\ell}\phi(\xi).
\end{align*}
Using the product rule, we arrive at
\begin{align*}
 D^i h_{q,\bseta,(r,\bsalpha,\beta)}
 &= (D^i T_1)\, T_2\,T_3
      + T_1\, (D^i T_2)\,T_3
      + T_1\, T_2\, (D^i T_3) \\
 &= \big(S_{1a} + S_{1b}\big)
 + S_2 + \sum_{m=1}^{r} \big(S_{3a,m} + S_{3b,m}\big),
\end{align*}
where
\begin{align*}
 &S_{1a} \coloneqq (r + q)\, h_{q,\bseta+\bse_i,(r+1, \widetilde\bsalpha,\beta)},
  &&\quad\mbox{with}\quad
  \widetilde\bsalpha_\ell
  \coloneqq
  \begin{cases}
  \bsalpha_\ell & \mbox{for } \ell=1,\ldots,r, \vspace{-0.1cm}\\
  \bse_i+\bse_0 & \mbox{for } \ell= r +1,
  \end{cases}
\\
  &S_{1b} \coloneqq (r+q)\, h_{q,\bseta+\bse_i,(r + 2, \widetilde\bsalpha,\beta)},
  &&\quad\mbox{with}\quad
  \widetilde\bsalpha_\ell
  \coloneqq
  \begin{cases}
  \bsalpha_\ell \quad\;\, & \mbox{for } \ell=1,\ldots,r, \vspace{-0.1cm}\\
  2\bse_0       \quad\;\, & \mbox{for } \ell=r+1, \vspace{-0.1cm}\\
  \bse_i        \quad\;\, & \mbox{for } \ell=r+2,
  \end{cases}
  \\
  &S_{2} \coloneqq h_{q,\bseta+\bse_i,(r+1, \widetilde\bsalpha,\beta+1)},
  &&\quad\mbox{with}\quad
  \widetilde\bsalpha_\ell
  \coloneqq
  \begin{cases}
  \bsalpha_\ell \qquad & \mbox{for } \ell=1,\ldots,r, \vspace{-0.1cm}\\
  \bse_i        \qquad & \mbox{for } \ell=r+1,
  \end{cases}
  \\
  &S_{3a,m} \coloneqq h_{q,\bseta+\bse_i,(r, \widetilde\bsalpha,\beta)},
  &&\quad\mbox{with}\quad
  \widetilde\bsalpha_\ell
  \coloneqq
  \begin{cases}
  \bsalpha_\ell & \mbox{for } \ell=1,\ldots,r,\,\ell\ne m,\vspace{-0.1cm}\\
  \bsalpha_m+\bse_i & \mbox{for } \ell=m,
  \end{cases}
  \\
  &S_{3b,m} \coloneqq h_{q,\bseta+\bse_i,(r+1, \widetilde\bsalpha,\beta)},
  &&\quad\mbox{with}\quad
  \widetilde\bsalpha_\ell
  \coloneqq
  \begin{cases}
  \bsalpha_\ell & \mbox{for } \ell=1,\ldots,r,\,\ell\ne m,\vspace{-0.1cm}\\
  \bsalpha_m+\bse_0 & \mbox{for } \ell=m,\vspace{-0.1cm}\\
  \bse_i & \mbox{for } \ell= r+1.
  \end{cases}
\end{align*}

Observe that all the $h_{q,\bseta+\bse_i,[\cdots]}$ functions above are of
the form \eqref{eq:h-form} with $\bseta$ replaced by $\bseta + \bse_i$,
and the conditions in \eqref{eq:h-form} are satisfied by an inductive
argument. For example, in $S_{1b}$, we gained two factors
$D^{2\bse_0}\phi(\xi)$ and $D^{\bse_i}\phi(\xi)$ to join the product over
$\ell$, increasing the upper limit of the product from $r$ to $r+2$, which
is consistent with the increase in the exponent of $D^0\phi(\xi)$ from
$r+q$ to $r+2+q$. Furthermore, $r+2 \le 2|\bseta|+q-1+2 =
2|\bseta+\bse_i|+q-1$, as required. Moreover, with
$\widetilde\bsalpha_{r+1} = 2\bse_0$ and $\widetilde\bsalpha_{r+2} =
\bse_i$, we have the updated sum $\beta \bse_0 + \sum_{\ell=1}^{r+2}
\widetilde\bsalpha_\ell = (r+q-1,\bseta) + 2\bse_0 + \bse_i =
(r+2+q-1,\bseta+\bse_i)$, as required. The result for the other terms
above can be justified in the same way. These $h_{q,\bseta+\bse_i,[\cdots]}$ functions are all
different so there is no cancellation.

Treating the multiple $(r+q) h_{q,\bseta+\bse_i,[\cdots]}$ in $S_{1a}$ as
$r+q$ occurrences of the same function and doing this analogously for
$S_{1b}$, we conclude that $D^i h_{q,\bseta,(r,\bsalpha,\beta)}$ can be
written as a sum of $K_{q,\bseta}$ functions of the form \eqref{eq:h-form}
with $\bseta$ replaced by $\bseta + \bse_i$, where
\begin{align*}
  K_{q,\bseta}
  &= (r+q) + (r+q)+1 + \textstyle\sum_{m=1}^r (1+1) = 4r + 2q+1 \\
  &\le 4(2|\bseta|+q-1) + 2q+1
  = 8|\bseta|+ 6q - 3.
\end{align*}
This completes the proof.
\end{proof}

\begin{lemma} \label{lem:count-t}
Let $d\ge 1$, $\bsnu\in \bbN_0^d$, and $[a,b]\subset\bbR$. Suppose that
$\phi$ and $\rho_0$ satisfy Assumption~\ref{asm:phi}, and recall the
definitions of $U_t$, $\xi$ and $V$ in \eqref{eq:Ut}, \eqref{eq:xi} and
\eqref{eq:V}, respectively. For any $q\in\bbN_0$ and $\bseta\le\bsnu$
satisfying $|\bseta|+q\le|\bsnu|+1$, we consider functions $h_{q,\bseta} :
V \to \R$ of the form \eqref{eq:h-form}. Then we can write
\begin{align*}
 \frac{\partial}{\partial t} h_{q,\bseta}(t,\bsy)
 = \sum_{k=1}^{K_{q,\bseta}} h_{q+1,\bseta}^{[k]}(t,\bsy),
 \quad\mbox{with}\quad
  K_{q,\bseta} \le 4|\bseta|+3q-1,
\end{align*}
where each function $h_{q+1,\bseta}^{[k]}$ is of the form
\eqref{eq:h-form} with $q$ replaced by $q+1$.
\end{lemma}

\begin{proof}
For $h_{q,\bseta}(t,\bsy) = h_{q,\bseta, (r,\bsalpha,\beta)}(t,\bsy)$
from \eqref{eq:h-form} we have
\begin{align*}
 \frac{\partial}{\partial t} h_{q,\bseta,(r,\bsalpha,\beta)}(t,\bsy)
 = \frac{\partial}{\partial t} \bigg(
     \underbrace{\frac{(-1)^r}{[D^0\phi(\xi(t,\bsy),\bsy)]^{r + q}}\vphantom{\prod_{\ell=1}^2}}_{=:\, T_1(t,\bsy)}\;
     \underbrace{\rho_0^{(\beta)}(\xi(t,\bsy)) \vphantom{\prod_{\ell=1}^2}}_{=:\, T_2(t,\bsy)}\;
     \underbrace{\prod_{\ell=1}^{r} D^{\bsalpha_\ell} \phi(\xi(t, \bsy), \bsy)}_{=:\, T_3(t,\bsy)}
     \bigg).
\end{align*}
Using the chain rule and substituting $\frac{\partial}{\partial t}\xi(t,\bsy)
= 1/D^0\phi(\xi(t,\bsy),\bsy)$ (see \eqref{eq:dt-xi}), and then
simplifying our notation by suppressing the dependence on $t$ and $\bsy$,
we obtain
\begin{align*}
 \frac{\partial}{\partial t} T_1(t,\bsy)
 &= \frac{-(r+q)\,(-1)^r\,D^0D^0\phi(\xi(t,\bsy),\bsy)\, \frac{\partial}{\partial t}\xi(t,\bsy)}
      {[D^0\phi(\xi(t,\bsy),\bsy)]^{r+q+1}}
 = \frac{(r+q)\,(-1)^{r+1}\,D^{2\bse_0}\phi(\xi)}{[D^0\phi(\xi)]^{(r+1)+(q+1)}},
 \\
 \frac{\partial}{\partial t} T_2(t,\bsy)
 &= \rho_0^{(\beta+1)}(\xi(t,\bsy))\,\frac{\partial}{\partial t}\xi(t,\bsy)
 = \frac{\rho_0^{(\beta+1)}(\xi)}{D^0\phi(\xi)},
 \\
 \frac{\partial}{\partial t} T_3(t,\bsy)
 &= \sum_{m=1}^{r} D^0 D^{\bsalpha_m}\phi(\xi(t,\bsy),\bsy)\, \frac{\partial}{\partial t}\xi(t,\bsy)
 \prod_{\substack{\ell= 1\\\ell\ne m}}^{r}\! D^{\bsalpha_\ell}\phi(\xi(t,\bsy),\bsy) \\
 &= \sum_{m=1}^{r} \frac{D^{\bsalpha_m+\bse_0}\phi(\xi)}{D^0\phi(\xi)}
 \prod_{\substack{\ell= 1\\\ell\ne m}}^{r}\! D^{\bsalpha_\ell}\phi(\xi).
\end{align*}
Using the product rule, we arrive at
\begin{align*}
 \frac{\partial}{\partial t} h_{q,\bseta,(r,\bsalpha,\beta)}
 &= \frac{\partial T_1}{\partial t}\, T_2\,T_3
      + T_1\, \frac{\partial T_2}{\partial t} \,T_3
      + T_1\, T_2\, \frac{\partial T_3}{\partial t}
 = S_1 + S_2 + \sum_{m=1}^{r} S_{3,m},
\end{align*}
where now
\begin{align*}
  &S_1 \coloneqq (r+q)\, h_{q+1,\bseta,(r+1, \widetilde\bsalpha,\beta)},
  &&\quad\mbox{with}\quad
  \widetilde\bsalpha_\ell
  \coloneqq
  \begin{cases}
  \bsalpha_\ell & \mbox{for } \ell=1,\ldots,r, \vspace{-0.1cm}\\
  2\bse_0 & \mbox{for } \ell=r+1,
  \end{cases}
  \\
  &S_2 \coloneqq h_{q+1,\bseta,(r, \bsalpha,\beta+1)},
  \\
  &S_{3,m} \coloneqq h_{q+1,\bseta,(r, \widetilde\bsalpha,\beta)},
  &&\quad\mbox{with}\quad
  \widetilde\bsalpha_\ell
  \coloneqq
  \begin{cases}
  \bsalpha_\ell & \mbox{for } \ell=1,\ldots,r,\,\ell\ne m,\vspace{-0.1cm}\\
  \bsalpha_m + \bse_0 & \mbox{for } \ell=m.
  \end{cases}
\end{align*}

Again, all of the $h_{q+1,\bseta,[\cdots]}$ functions above are of the
form \eqref{eq:h-form} with $q$ replaced by $q+1$, and the conditions in
\eqref{eq:h-form} are satisfied by an inductive argument with justification
similar to the arguments in the proof of Lemma~\ref{lem:count-y}.
These $h_{q+1,\bseta,[\cdots]}$ functions are all different so there is no
cancellation.

Treating the multiple $(r+q) h_{q+1,\bseta,[\cdots]}$ in $S_1$ as $r+q$
occurrences of the same function, we conclude that
$\frac{\partial}{\partial t} h_{q,\bseta,(r,\bsalpha,\beta)}$ can be
written as a sum of $K_{q,\bseta}$ functions of the form \eqref{eq:h-form}
with $q$ replaced by $q+1$, where now
\begin{align*}
  K_{q,\bseta}
  &= (r+q)+1 + \textstyle\sum_{m=1}^r 1 = 2r + q+1
  \le 2(2|\bseta|+q-1) + q+1 = 4|\bseta| + 3q-1.
\end{align*}
This completes the proof.
\end{proof}

\end{appendix}

\end{document}